\documentclass[11pt,a4paper]{amsart}
\usepackage[margin=2.5cm]{geometry}

\usepackage[usenames,dvipsnames]{color}
\usepackage{hyperref}
\hypersetup{
 colorlinks,
 linkcolor={red!50!black},
 citecolor={blue!50!black},
 urlcolor={blue!80!black}
}
\usepackage{url}
\usepackage{tikz}
\usepackage{pdflscape}
\usepackage{longtable}
\usepackage{booktabs}
\usepackage{colortbl}
\usepackage{verbatim}
\newcommand{\evnrow}{\rowcolor[gray]{0.95}}

\newcommand{\C}{\mathbb{C}}
\newcommand{\Z}{\mathbb{Z}}
\newcommand{\PP}{\mathbb{P}}
\renewcommand{\P}{\mathbb{P}}
\newcommand{\Q}{\mathbb{Q}}
\newcommand{\cM}{\mathcal{M}}
\newcommand{\cB}{\mathcal{B}}
\newcommand{\cO}{\mathcal{O}}
\newcommand{\lineq}{\buildrel{\mathrm{lin}}\over{\sim}}
\DeclareMathOperator{\prim}{prim}

\DeclareMathOperator{\I}{I}

\DeclareMathOperator{\IV}{IV}
\newcommand{\of}{\mathcal{O}}
\DeclareMathOperator{\Ext}{Ext}

\DeclareMathOperator{\codim}{codim}
\DeclareMathOperator{\Gr}{Gr}
\DeclareMathOperator{\Co}{C}
\newcommand{\eorb}{e_\mathrm{orb}}
\newcommand{\corb}{c_\mathrm{orb}}
\newcommand{\corbn}{c_{n,\mathrm{orb}}}

\newcommand{\Ybar}{\overline{Y}}

\newcommand{\tf}{3\nobreakdash-\hspace{0pt}fold}
\newcommand{\wGr}{w\Gr}
\newcommand{\wP}{w\P}
\newcommand{\grdb}{{\sc{Grdb}}}

\newtheorem{thm}{Theorem}
\newtheorem*{thm*}{Theorem}

\newtheorem{lemma}[thm]{Lemma}

\theoremstyle{definition}

\newtheorem{example}[thm]{Example}
\theoremstyle{remark}
\newtheorem{remark}{Remark}

\begin{document}

\title{Hodge numbers and deformations of Fano 3-folds}
\author{Gavin Brown and Enrico Fatighenti}
\date{}
\begin{abstract}
We calculate the Hodge numbers of quasismooth
Fano 3-folds whose total anticanonical embedding
has small codimension, and relate these to
the number of deformations.
\end{abstract}
\maketitle


\section{Introduction}

In this paper we calculate the Hodge numbers and the number of moduli
of all known (index~1) Fano 3-folds in codimensions~1, 2 and~3.
These results are presented in Tables~\ref{tab!codim1}--\ref{tab!codim3} respectively;
the Picard rank is~1 in every case.
We also calculate a few cases in codimensions~4 in \S\ref{s!codim4},
where the Picard rank is sometimes larger.

A {\em Fano 3-fold} is a normal 3-dimensional complex
projective variety $X$ with ample anticanonical class $-K_X$
and $\Q$-factorial terminal singularities, and
we restrict consideration to those $X$ whose singularities
are terminal cyclic quotient singularities. Any such $X$ is a
projective orbifold, the quotient of a projective manifold
by a finite cyclic group.
It is known by Sano \cite{sano} that any Fano 3-fold
has a small deformation (a {\em $Q$-smoothing})
that has only quotient singularities.

A {\em K3 elephant} of a Fano 3-fold $X$ is an irreducible
surface $E\subset X$ with canonical singularities
that is linearly equivalent to ${-}K_X$.
In particular, $E$ has $K_E=0$, and so $E$ is a K3 surface.

There are two main ingredients.
The first is an unprojection calculus (see \S\ref{s!unproj} or \cite{TJ}).
The second is a relation betwen the Hodge numbers of a
Fano 3-fold and the number of its moduli, together with an
infinitesimal rigidity result, which we summarise as follows.

\begin{thm}\label{th!main}
Let $X$ be a Fano 3-fold with K3 elephant $E\subset X$
and genus $g_X = h^0(X,-K_X)-2$.
\begin{enumerate}
\item
\label{th!main1}
Setting $\alpha_E= h^{1,1}(E)-g_X+1$,
\begin{equation}\label{eq!moduli}
h^1(X,T_X)-h^0(X,T_X) =  \alpha_{E} + h^{2,1}(X) - h^{2,2}(X).
\end{equation}
\item
\label{th!main2}
If $X$ is a complete intersection in weighted projective space
or in a weighted Grassmannian $w\Gr(2,5)$, then $h^0(X,T_X)=0$.
\end{enumerate}
\end{thm}
Part~\eqref{th!main1} is proved in \S\ref{s!def} 
and part~\eqref{th!main2} in \S\ref{s!auto}.
We work over~$\C$ throughout.

\section{Preliminaries}

\subsection{Fano 3-folds in their anticanonical embeddings}

We study a Fano 3-fold $X$ using its anticanonical graded ring
\[
R(X,-K_X) = \bigoplus_{m\ge0} H^0(X,\cO_X(-mK_X)).
\]
A minimal set of generators $x_0,\dots,x_n$ for $R(X,-K_X)$, whose degrees are
denoted $a_0,\dots,a_n$,
present $X$ as a subvariety $X\subset\P(a_0,\dots,a_n)$
defined by the relations holding in the ring.
By definition, the {\em codimension} of a Fano 3-fold $X$
is its codimension in this embedding: $\codim(X)=n-3$.
Such embedded $X$ is an orbifold if its equations satisfy
the Jacobian condition.

According to \cite{KMM} (following \cite{K92} in the case of Mori--Fano 3-folds),
the classification of Fano 3-folds consists of finitely many
deformation familes.
The Hilbert series of members of those families whose generic
element lies in codimension at most~4 are known \cite{Aphd,ABR02}
and available on the Graded Ring Database \cite{BK}.
They fall into $95+85+70+145 = 395$ cases, according to codimension.
There may be more than one irreducible family
for any given Hilbert series, and in codimension~4
there are usually two or more families in each case \cite{TJ};
the different families are distinguished by the Euler characteristic
of their general member.

\subsection{The Hodge numbers of Fano 3-folds}
\label{s!intro1.1}

On a quasi-smooth variety $X$ is it possible to define the notion of a pure Hodge structure,  see Steenbrink \cite[Theorem~1.12]{steenbrink}. Consider in fact the smooth locus $j:X_0\hookrightarrow X$ and $\widehat{\Omega}^p_X:=j_* \Omega^p_{X_0}$. Then we can define $H^{p,q}(X)$ as in the smooth case and moreover $H^{p,q}(X) \cong H^q(X, \widehat{\Omega}^p_X)$. The Hodge decomposition takes then the form $$H^k(X,\C) = \bigoplus_{p+q=k}   H^q(X, \widehat{\Omega}^p_X).$$
  Since there will be no danger of confusion, to avoide cumbersome notations when dealing with quasi-smooth varieties we will abuse the notation and write directly $\Omega^p_X$ instead of $\widehat{\Omega}^p_X$. It follows at once from the Lefschetz hyperplane theorem and
Kawamata--Viehweg vanishing that the Hodge diamond
of a Fano  3-fold $X$ has the form
\[
\begin{array}{ccccccc}
&&& h^{3,3} \\
&& h^{3,2} && h^{2,3} \\
& h^{3,1} && h^{2,2} && h^{1,3} \\
h^{3,0} && h^{2,1} && h^{1,2} && h^{0,3} \\
& h^{2,0} && h^{1,1} && h^{0,2} \\
&& h^{1,0} && h^{0,1} \\
&&& h^{0,0}
\end{array} =
\begin{array}{ccccccc}
&&& 1 \\
&& 0 && 0 \\
& 0 && h^{2,2} && 0 \\
0\  && h^{2,1} && h^{1,2} && \ 0 \\
& 0 && h^{1,1} && 0 \\
&& 0 && 0 \\
&&& 1
\end{array}.
\]
The Euler characteristic $e(X)$ of $X$ can be expressed as
\[
e(X) = 2 + 2h^{1,1}(X) - 2h^{2,1}(X).
\]
We calculate these three integers for
Fano 3-folds $X$ lying in the known families of Fano
3-folds with small codimension.
We explain the different strategies we employ
in \S\ref{sec!strategies} below.

The answer is well known in codimension~1:
the Hodge numbers of weighted hypersurfaces are computed
by results of Griffiths, Dolgachev and Dimca.
(Recall that primitive cohomology is the kernel of the
hyperplane operator: if $X$ has dimension $n$ and
hyperplane class $L$, then
\[
H^k(X,\C)_{\prim} = \ker \left\{ \cap L^{n-k+1}\colon H^k(X,\C) \rightarrow H^{2n+2-k}(X,\C) \right\},
\]
and $H^{p,q}_{\prim}(X) = H^{p,q}(X) \cap H^{p+q}(X,\C)_{\prim}$.
When $X$ is a Fano 3-fold, then  $b_5(X) = 0$ and so $H^{2,1}_{\prim}(X) = H^{2,1}(X)$.)

\begin{thm*}[\cite{dolgachev,dimca,fletcher}]
Let $X_d\colon (f=0)\subset\P(a_0, \ldots, a_n)$ be a quasismooth
hypersurface, defined by a homogeneous polynomial~$f$ of degree~$d$
in weighted homogeneous coordinates
$x_0,\dots,x_n$ of degrees $\deg x_i=a_i$.
Then the Milnor algebra $\cM$ of $X$ is $\C[x_0, \ldots, x_n] /J_f$
is finite dimensional, and there is an isomorphism
\[
H^{n-p, p-1}_{\prim}(X) \cong \cM^{pd-\sum a_i}.
\]
\end{thm*}
The Hilbert Series $P_\cM$ of the Milnor algebra $\cM$ is given, in the notation
of the theorem, by
\[
P_\cM=\dfrac{(1-t^{b_0}) \cdots (1-t^{b_n})}{(1-t^{a_0}) \cdots (1-t^{a_n}) },
\quad\text{where }b_i= d- a_i.
\]
For example, $X_{66}\subset\P(1,5,6,22,33)$ has
\begin{eqnarray*}
P_\cM &=& \frac{\prod_{b\in\{65,61,60,44,33\}} (1-t^b)}
{\prod_{a\in\{1,5,6,22,33\}} (1-t^a)} \\
 &=& 1 + t + t^2 + t^3 + t^4 + 2t^5 + 
 \cdots + 118t^{64} + 120t^{65} + 122t^{66} + \cdots + t^{196}.
\end{eqnarray*}
Thus we read off
$h^{2,1}(X) = \dim \cM^{2\cdot 66 - 67} = \dim\cM^{65} = 120$.
We list all 95 cases in Table~\ref{tab!codim1}.

In codimensions 2, 3 and 4, the Euler characteristic is known in most
cases by \cite{TJ}, so knowing $h^{1,1}(X)$ completes the calculation.
Blache \cite{blache} describes a general theory of orbifold characteristic
classes, and their relations with the usual topological notions, that
we describe and compare in Appendix~\ref{s!blache}. 
We calculate codimension~2 using our methods described
below, and Theorem~\ref{th!main} is crucial in the higher
codimension, non-complete intersection cases---and the cases
with higher Picard rank in \S\ref{s!codim4} use these in an
essential way.
Thus the first observation is that this is readily computed
in low codimension, since every Fano 3-fold in codimension up to~3
appears in one of the two situations of the theorem.

\begin{thm}
\label{thm!lefschetz}
If $X$ is a quasismooth Fano 3-fold that is either
\begin{enumerate}
\item\label{lefi}
a complete intersection in weighted projective space, or
\item\label{lefii}
a complete intersection in a weighted cone over a weighted $\Gr(2,5)$,
\end{enumerate}
then $h^{1,1}(X) = 1$.
\end{thm}

\begin{proof}
We prove that $T^2_{A_X}(-1)=0$, where $A_X$ is the affine cone on $X$.
This is enough since $H^2(X,K_X)=0$ allows us to apply
\cite[Theorem 2.8]{io}, which says $H^{1,1}_{\prim}(X) = T^2_{A_X}(-1) = 0$,
and so $h^{1,1}(X)=1$.

In part \eqref{lefi}, the vanishing is \cite[1.3]{schlessinger}. For part \eqref{lefii}, $ T^2_{A_X}(-1) \cong H^1(X, N_{X/ \Co \PP}(-1))$, where $\Co\PP$ denotes the ambient projective space for the Grassmannian in its Pl\"ucker embedding with the addition of the cone variables. From \cite[\S{}D.1, Lemma~D.3]{sernesi2007deformations} the flag of schemes
$X \subset \Co \Gr \subset \Co \PP$ determines a sequence of sheaves on $X$:
$$ 0 \to N_{X/\Co \Gr} \to N_{X/\Co \PP} \to N_{\Co\Gr/ \Co\PP} \to 0,$$ where the last map is exact since $H^1(N_{X/\Co \Gr})=0$.
Twisting by $\of_X(-1)$ we get $$H^1 (N_{X/\Co \PP}(-1)) \cong H^1(N_{\Co\Gr/ \Co\PP}(-1))=0.$$ This proves part \eqref{lefii}.
\end{proof}

Part~\eqref{lefi} of this result appeared in a recent preprint, \cite{pizzato2017effective}, and we found \eqref{lefii} stated several times in the literature, such as \cite{kim2016alpha}, but we could not find a proof to cite.
In this situation, one would like appeal to folklore and simply apply a weighted Lefschetz
hyperplane theorem for ample systems.
But unfortunately the linear systems we cut
by to make $X$ are rarely base-point free when there are nontrivial weights,
so the strong results in the literature
such as \cite[Theorem~1]{RS} and \cite[Corollary~2.8]{HL} do not apply directly.

\subsection{Fano 3-folds and projection}
\label{s!unproj}

Consider the following arrangement of projective 3-folds:
\begin{equation}\label{pic!unproj}
\begin{array}{ccccc}
 && \widetilde{Y} & \rightarrow & X \\
 && \downarrow  \\
Y & \rightsquigarrow & \overline{Y}
\end{array}
\end{equation}
where $X$ and $Y$ are quasismooth, $Y\rightsquigarrow \overline{Y}$
is a degeneration to a singular orbifold whose only non-quasismooth points are ordinary nodes,
$ \overline{Y}\leftarrow\widetilde{Y}$ is a projective small resolution
of the nodes, and $\widetilde{Y}\rightarrow X$
is the contraction of a divisor $\widetilde{D}\subset\widetilde{Y}$.
The passage from $Y$ to $\widetilde{Y}$, that shrinks a number of
vanishing cycles to nodes and then resolves the nodes by exceptional $\P^1$s,
is known as a {\em conifold transition}.

In our context, the exceptional divisor $\widetilde{D}\cong\P(a,b,c)$
maps to a divisor $\P(a,b,c)\rightarrow D\subset\overline{Y}$,
and the nodes of $\overline{Y}$ lie on $D$.
The small resolution is the relatively $\widetilde{D}$-ample
resolution, so is projective, and $\widetilde{D}\rightarrow D$
is birational---often an isomorphism, in fact.
With this setup, we recall from Clemens~\cite{clemens} (see also~\cite[\S5]{reidCY}):
\begin{thm}[\cite{clemens,reidCY}]
\label{th!nodes}
Let $X$ and $Y$ be Fano 3-folds related as in diagram~\eqref{pic!unproj}. Then
\begin{equation}
\label{eq!e}
e(X)=e(Y)+2n-2,
\end{equation}
where $n$ is the number of nodes of $\overline{Y}$.
In particular,  if $h^{1,1}(X) = h^{1,1}(Y)$, then
\begin{equation}
\label{eq!h21}
h^{2,1}(X) = h^{2,1}(Y) - n + 1.
\end{equation}
\end{thm}

The relevance of this is as follows (see \cite[2.6.3]{CPR}, \cite[3.2]{TJ}).
If $X$ is a Fano 3-fold in codimension~$k$,
then it often happens that the {\em Gorenstein projection} from a quotient singularity
sits in diagram~\eqref{pic!unproj} as $X\dashrightarrow\overline{Y}$,
and that $\overline{Y}$ lies in codimension $<k$.
If this nodal Fano $\overline{Y}$ deforms to a quasismooth $Y$ whose
Hodge numbers are known, then we may recover the invariants of~$X$.

\subsection{An overview of the calculations }
\label{sec!strategies}

We adopt different tactics to compute the Hodge numbers
of a Fano 3-fold $X$ according to its graded ring.

\subsubsection{}
When $X$ is a hypersurface, this calculation is well known (see \S\ref{s!intro1.1}).

\subsubsection{}
If $X$ is a complete intersection in weighted projective
space or inside a weighted Grassmannian,
then $h^{1,1}(X) = 1$ (Theorem~\ref{thm!lefschetz}).
If $X$ arises by (possibly multiple) unprojection from
a hypersurface, then we can compute $e(X)$
and hence the whole Hodge diamond.
This applies to most $X$ that lie in codimension~2 or~3;
see \S\S\ref{s!codim2}--\ref{s!codim3}.
Up to codimension~3, this calculation can be done by
hand---the key point is to confirm the existence
of a nodal degeneration.

\subsubsection{}
Denoting the affine cone over $X$ by $A_X$,
\cite[Theorem~2.5]{io} gives
\[
H^{2,1}(X) \cong T_{A_X}^1(-1).
\]
If $X$ is given by explicit equations, we may use standard algorithms
and implementations in computer algebra to
calculate $h^{2,1}(X)$; see \S\ref{s!T1} and \S\ref{s!comp}.

In these cases we compute
$h^{2,1}(X)$ for a single quasismooth member of
each family, expressed in the format we expect.
Since $h^{p,q}$ are deformation invariants for orbifolds (since
Steenbrink \cite[Theorem 2]{steenbrink}
applies in the context of V-manifolds),
the numbers we obtain are also the Hodge numbers of any
orbifold Fano 3-fold in the family.

\subsubsection{}
By \cite[Theorem~2.8]{io},
\[
H^{1,1}_{\prim}(X)(X) \cong T_{A_X}^2(-1),
\]
and so if $X$ is given by explicit equations we may
compute $h^{1,1}(X)$; see Section~\ref{s!codim4} for an example.
This algorithm seems to be more complicated, and in
practice choosing good equations is delicate.

\subsection{Calculating $T^1$ and $h^{2,1}(X)$ by computer algebra}
\label{s!T1}

We recall the context and results of \cite{io}.
A {\em subcanonical pair} $(X, \of_X(1))$ consists of a
quasismooth projective variety $X$ and an ample sheaf $\of_X(1)$
which satisfies $\omega_X \cong \of_X(k_X)$ for some $k_X \in \Z$.

Let $(X,\of_X(1))$ be a subcanonical pair.
We denote by $A_X$ the affine cone over $X$
and by $U_X= A_X \setminus \{v\}$, where $v$ is the vertex of the cone.
The results of \cite{io} require that $\mathrm{depth}_v A_X \geq 3$,
which holds in our context since $H^1(X,\cO_X(j))=0$ for any $j\in\Z$.

Consider the space $T^1_{A_X}$ that parametrizes the set of isomorphism classes of first order infinitesimal deformations of $A_X$. This is defined (as in \cite{schlessinger}, since $X$ is projectively normal) by
\[
T^1_{A_X}:= \Ext^1(\Omega^1_{A_X}, \of_{A_X}),
\]
and admits a natural $\Z$-grading given by the natural $\C^*$-action on $A_X$.

By \cite{schlessinger}, the degree 0 component of the deformations of
the affine cone parametrizes the embedded deformations of $X$;
that is, the deformations of the pair $(X, \of_X(1))$.
Furthermore, the negative components are identified with the smoothings
of the affine cone, while the positive components
parametrize equisingular deformations.
In the case of a smooth projective hypersurface of degree $d$,
 $$ T^1_{A_X}(-d) \cong \C[x_0, \ldots, x_n]/J_f,$$
the Jacobian ring of $X$, as in \S\ref{s!intro1.1}.

\begin{thm}[\cite{io} Theorem 1.1]
Let $(X,\of_X(1))$ be a subcanonical pair with $\omega_X \cong \of_X(k_X)$.
Set $n=\dim X$.
Then there is an isomorphism
\[
T^1_{A_X}(k) \cong
\ker\left(\lambda\colon
   H^1 (X, \Omega^{n-1}(k-k_X)) \longrightarrow H^2(X, \omega_X(k-k_X)\right),
\]
where $\lambda(\eta)= c_1(\of_X(1)) \wedge \eta$.
\end{thm}

When $k=k_X$, the statement becomes
$T^1_{A_X}(k_X) \cong H^{n-1,1}_{\prim}(X)$, the primitive cohomology.

\section{Moduli of Fano 3-folds}
\label{s!moduli}

We explain a relation between $H^{2,1}(X)$ of a Fano threefold $X$
and the tangent space to its versal deformation space $H^1(X, T_X)$.
Since deformations of quasismooth Fano 3-folds $X$ are unobstructed
(by \cite[Theorem 1.7]{sano}), this is the number of moduli of $X$.

\subsection{Deforming a Fano with an elephant}
\label{s!def}

The idea comes from Calabi--Yau \tf{}s. Given such a $V$,
it  follows by standard Serre duality
(non-canonically, involving a choice of determinant) that
$H^{2,1}(V)\cong H^1(V,T_V)$;
or one may observe that both are isomorphic
to the same graded piece $T^1_{A_V}(0)\subset T^1_{A_V}$.

If a Fano 3-fold $X$ has a K3 elephant $E = (x=0) \subset X$,
we may regard the pair $(X,E)$ as a log Calabi--Yau and hope to
mimic this relationship.
In the index~1 case, one has
$H^{2,1}(X) \cong T^1_{A_X}(-1)$ and $H^1(X,T_X) \cong T^1_{A_X}(0)$,
and the analogue to the Calabi--Yau isomorphism is the
multiplication map $x\colon H^{2,1}(X)\rightarrow H^1(X,T_X)$.
This map is not an isomorphism, in general,
but Theorem~\ref{th!alpha} below explains the difference in terms of
the geometry of~$E$.
To make this intuition precise, we start with a more general lemma
about Fano 3-folds of arbitrary index~$m>0$.
\begin{lemma}\label{lem!alpha}
Let $X$ a Fano threefold.
If $E \subset X$ a K3 elephant $E \in | {-}K_X |$, then
\[
h^1(X,T_X)-h^0(X,T_X)= \alpha_{E} + h^{2,1}(X)-h^{2,2}(X),
\]
where $\alpha_E= h^{1,1}(E)-g_X+1$.
\end{lemma}

\begin{proof}
Suppose $X$ is of index m with $-K_X \lineq mH$,
for an ample $\Q$-Cartier divisor~$H$.

Consider the standard exact sequence of $\cO_X$-modules
twisted by $\Omega^2(m)$,
\[ 0 \to \Omega^2_X \to \Omega^2_X(m) \to \Omega^2_X(m)|_E \to 0. \]
In cohomology this yields a long exact sequence
\begin{equation}
\begin{split}
\label{exact1}
0 \to H^0(\Omega^2_X(m)) & \to  H^0( \Omega^2(m)_X|_E)
 \to H^1(\Omega^2_X)  \\
 & \to H^1(\Omega^2_X(m))
  \to H^1(\Omega^2_X(m)|_E)\to H^2(\Omega^2_X) \to 0,
\end{split}
\end{equation}
where $H^0(\Omega^2_X)=0$ and $H^2(\Omega^2_X(m))=0$
by Akizuki--Kodaira--Nakano vanishing.

On the other hand the relative exact tangent sequence \[ 0 \to T_E \to T_X|_E \to \of_E(m)\to 0\] yields a long exact sequence
\begin{equation}\label{exact2}
0\to H^0(E,T_X|_E) \to H^0(E, \of_E(m)) \to H^1(E,T_E) \to H^1(E,T_X|_E) \to 0,
\end{equation}
where $H^1(E, \of_E(m))=0$ and $H^0(E, T_E)= H^0(E, \Omega^1_E)=0$, since $E$
is K3 surface.
By \eqref{exact1} and \eqref{exact2} we get
\begin{equation}
\begin{split}
h^0(X,\Omega^2_X|_E(m)) & + h^1(X,\Omega^2_X(m)) + h^{2,2}(X) = \\
 & h^{2,1}(X) + h^{1}(X,\Omega^2_X(m)|_E) + h^0((X,\Omega^2_X(m))
\end{split}
\end{equation}
and
\[h^1(T_X|_E)-h^0(T_X|_E)= h^1(T_E)-h^{0}(\of_E(m)). \]

We have $\Omega_X^2 (m)\cong T_X$ from the pairing
\[ \Omega^1_X \otimes \Omega^2_X \to \omega_X \cong \of_X(-m).\]
So with $\alpha_E$ defined as in the statement,
we get
 \[
 h^1(X,T_X)-h^0(X,T_X)= \alpha_{E} + h^{2,1}(X)-h^{2,2}(X)
 \]
 as required.
\end{proof}

\begin{thm}\label{th!alpha}
Let $X$ be a Fano 3-fold with K3 elephant $E\subset X$
and $\alpha_E$ as defined in  Lemma~\ref{lem!alpha}.
If $h^0(X,T_X)=0$, then
\[
h^1(X,T_X)-h^{2,1}(X)= \alpha_{E} -h^{2,2}(X).
\]
\end{thm}
This gives an estimate of the difference between the
moduli and Hodge theory of $X$: when $b_2 = h^{2,2}(X)$ is small,
we have a more moduli than $h^{2,1}$, while if  $b_2 >\!\!>0$ the opposite holds.

\begin{remark}
The number $\alpha_E$ is a function of the polarised K3 surface~$E$
(since $h^0(E, \of_E(E))=g_X-1)$.
When $E$ is smooth $h^{1,1}(E)=20$, and so
$\alpha_E= 20- h^0(E, \of_E(1))$. More generally, if $E$ has canonical
singularities with corresponding basket $\cB = \left\{ \frac{1}r(a,-a) \right\}$
(see \cite[Theorem (9.1)(III)]{YPG}), then
\[
\alpha_E = 20 - \sum_\cB (r-1) - h^0(E, \of_E(m)).
\]
In  every case we know,
when a general member $X$ of a family of
Fano 3-folds has a K3 elephant $E\subset X$,
then both $X$ and $E$ are quasismooth; in particular,
they both have only quotient singularities, and the basket of $E$
is equal to the set of singularities of $E$.
\end{remark}

\subsection{Automorphisms of Fano 3-folds in Grassmannians}
\label{s!auto}

\begin{lemma}
Let $X$ be a Fano 3-fold of index~1.
If $X$ is a weighted complete intersection (in its total anticanonical
embedding), then $H^0 (X, T_X)=0$.
\end{lemma}
\begin{proof}
Recall from Flenner \cite[Satz~8.11]{flenner81} that if
$X$ is an $n$-dimensional weighted complete intersection,
then
$H^p(X, \Omega_X^q(t))=0$ whenever $p+q<\dim X$ and $t<q-p$. 

The lemma follows by setting $q=2$, $p=0$, $t=1$ together with
Serre duality $T_X \cong \Omega^2_X(1)$.
\end{proof}

We prove an analogous result for complete intersection in weighted Grassmannians.
Our main interest is in Fano 3-folds of index~1 in codimension~3, $X\subset\P(a_0,\dots,a_6)$,
most of which arise in this way.
We show in Theorem~\ref{cor!grass} below that $H^0(X, T_X)=0$ in this case.
We first show the vanishing result in standard (non-weighted) Grassmannians.

\begin{lemma}\label{claim} Let $X$ a Fano 3-fold of index 1 that is a complete intersection in a cone $V = C\Gr(2,n)$, on vertex a linear projective space that is disjoint from $X$, over a Grassmannian $\Gr(2,n)$ for some $n\ge 5$. Then $H^0(X, T_X)=0$.
\end{lemma}
\begin{proof}
We show that $H^0(X, \Omega^2_X(1))=0$, which suffices
since $T_X\cong \Omega_X^2(1)$ for $X$ a Fano 3-fold of index~1.

We consider the case $V=\Gr(2,n)$ first, with no cone structure.
Suppose that $X = (f_1=\cdots=f_c=0)\subset G = \Gr(2,n)$,
and denote $d_i=\deg f_i$.
The Koszul complex of $\cO_X$-modules
for $\of_X$ twisted by $\Omega^2(1)|_G$ is
\[ 
\begin{split}
0 \to \Omega^2_G&(1-d_1-\cdots - d_c) \to \cdots \to \bigoplus_{i,j,k} \Omega^2_G(1-d_i-d_i-d_k) \to \\
 & \bigoplus_{i,j} \Omega^2_G(1-d_i-d_j) \to \bigoplus_i \Omega^2_G(1-d_i) \to \Omega^2_G(1) \to \Omega^2_G(1)|_X \to 0.
\end{split}
\]
By \cite[Lemma~0.1]{peternell}, $H^p(G,\Omega_G^2(t))=0$
for each of $p=1,2,3$ and any $t\le 1$, and also $H^0(G,\Omega_G^2(1))=0$.
It follows, by splitting the Koszul sequence above into short exact sequences,
that
\begin{equation}
\label{eq!van}
H^0(X,\Omega^2_G(1)|_X) = H^1(X,\Omega^2_G(1)|_X) =
H^1(X,\Omega^2_G(1-d_i)|_X)=0.
\end{equation}

The conormal exact sequence of $X\subset G$ is
\[
0 \to \bigoplus_{1\le i\le c} \of_X(-d_i) \to \Omega^1_G|_X \to \Omega^1_X \to 0.
\]
Taking its second exterior power and twisting by $\of_X(1)$ we get
\[
0 \to \bigoplus_{1\le i,j\le c}  \of_X(1-d_id_j) \to \bigoplus_{1\le i\le c}  \Omega^2_G(1-d_i)|_X \to \Omega^2_G(1)|_X \to \Omega^2_X(1)\to 0.
\]
After splitting this into short exact sequences, the vanishing
statements in \eqref{eq!van} show at once that
$H^0(X, \Omega^2_X(1))=0$, as required.

The proof for a cone is the same, replacing $\Omega^2_{\Gr}$ by the
extension of the pullback of $\Omega^2_{\Gr}$ to the complement of the vertex,
in which $X$ is a complete intersection;
this restricts to $X$ as above, and the proof follows.
\end{proof}

The proof of Lemma~\ref{claim} suggests that we need a Bott vanishing
type of result to extend the vanishing statements to complete intersections
in $w\Gr(2,5)$. The following lemma gives the precise statement we need.
\begin{lemma}
Let $wG=wGr(2,5)$. Then $H^p(w\Gr, \Omega^2_{w\Gr}(-k))=0$
for $p=1,2,3$ and any $k>0$.
\end{lemma}
\begin{proof}
If $A_G^\bullet$ denotes the punctured affine cone over the (weighted or not) Grassmannian, we have the following diagram
\[
\begin{array}{ccccccc}
 &&&&\  A_G^\bullet \\
 &&& \pi_1\swarrow && \searrow \pi_2 \\
 & & Gr(2,5) &&&&wGr(2,5)
\end{array}
\]
where $\pi_1$ and $\pi_2$ denote the quotients by the standard and the weighted $\C^*$ actions respectively. We use the vanishing results from \cite[Lemma 0.1]{peternell} for the standard $\Gr(2,5)$ repeatedly.

The grading on the cohomology groups of $A^{\bullet}$ is interpreted in terms of local cohomology at the maximal ideal $\mathfrak{m}$ of the vertex of the affine cone $A$.

Consider the short exact sequence
\begin{equation}\label{eq!omega1}
0 \to \pi_1^*\Omega^1_G \to \Omega^1_{A^\bullet} \to \of_{A^ \bullet} \to 0.
\end{equation}
Since $H^i(G, \of_G(-k))=0$ for any $i<\dim(G)$, we have $$H^1(A^\bullet, \Omega^1_{A^\bullet})(-k)=H^1(G, \Omega^1_G(-k))=0.$$
In the same way one also gets $H^0(A^\bullet,\Omega^1_{A^\bullet})(-k)=0$.

Raising the short exact sequence \eqref{eq!omega1} to the second exterior power we have $$ 0 \to \pi_1^* \Omega^2_G \to \Omega^2_{A^\bullet} \to \pi_1^* \Omega^1_G \to 0 ;$$ by the vanishing statements above this reduces to $$H^1 (A^\bullet,\Omega^2_{A^\bullet})(-k)= H^1(G,\Omega^2_G(-k))=0.$$ 
 Comsidering analogous exact sequences for the second projection $\pi_2$ gives
 $$ 0 \to \pi_2^*\Omega^1_{wG} \to \Omega^1_{A^\bullet} \to \of_{A^ \bullet} \to 0, $$ 
 $$ 0 \to \pi_2^* \Omega^2_{wG} \to \Omega^2_{A^\bullet} \to \pi_2^* \Omega^1_{wG} \to 0 .$$
Putting all these vanishing statements together with $H^0(\of_{wG}(-k))=0$ we get
$$H^1(wG, \Omega^2_{wG}(-k))=H^1(A^{\bullet}, \Omega^2_{A^\bullet})(-k)=0,$$
as required. The results for $i=2,3$ follow similarly.
\end{proof}
 
\begin{thm}\label{cor!grass}
Let $X$ a Fano 3-fold of index 1 that is a complete intersection in a weighted cone $C\Gr(2,5)$,
with vertex a linearly-embedded weighted projective space that is disjoint from~$X$.
Then $H^0(X, T_X)=0$.
\end{thm}

Both the lemma and the theorem can be extended to weighted Grassmannians 
$w\Gr(2,n)$, for $n\ge5$, using Bott-type vanishing theorems, but we only need
the $\Gr(2,5)$ case here.

\section{Explicit calculations}

It takes a few hundred calculations to complete Tables~\ref{tab!codim1}--\ref{tab!codim3}
below.
In this section, we give illustrative examples of each type.

\subsection{Codimension 2}
\label{s!codim2}

There are 85 deformation families of Fano 3-folds
in codimension~2 (\cite{fletcher,CCC}), each one
a complete intersection with $h^{1,1}(X) = 1$.
The case $X_{2,3}\subset\P^5$ is classical: $e(X)=c_3(T_X)$
can be calculated directly to give
$e(X_{2,3}) = -36$ and so $h^{2,1}(X_{2,3})=20$.
Of the remaining 84 cases, 66 have a Type~I projection (see \S\ref{s!c2t1}),
and a further 10 cases have a Type~II$_1$ projection (see \S\ref{s!c2y2});
8 cases have no projection of either type (see \S\ref{s!c2noproj}).

\subsubsection{66 cases with Type I projection}
\label{s!c2t1}

Consider one of the families of Fano 3-folds of the form
$X=X_{a_3+r,a_4+r}\subset\P(1,a,r-a,a_3,a_4,r)$
with $a<r$.
The general member has a quotient singularity
$\frac{1}r(1,a,r-a)$, and admits a Type~I projection,
as in diagram~\eqref{pic!unproj},
to a hypersurface:
\begin{eqnarray*}
X & \subset & \P(1,a,r-a,a_3,a_4,r) \\
 \pi_r\downarrow &&  \\
D\subset (x_3A = x_4B) = \Ybar
 & \subset & \P(1,a,r-a,a_3,a_4),
\end{eqnarray*}
where $D = (x_3 = x_4 = 0) = \P(1,a,r-a)$ and $\pi_r$
is the projection from the final coordinate point of index~$r$.
In each one of these 66 cases,
the general $\Ybar$ is quasismooth away
from $n=\deg(A)\deg(B) / (a(r-a))$ nodes that lie on $D$ (by Bertini's theorem),
and it admits a Q-smoothing to a general $Y=Y_{a_3+a_4+r}\subset\P(1,a,r-a,a_3,a_4)$.
Thus we calculate $e(X) = e(Y) + 2n - 2$ by~\eqref{eq!e}.

\begin{example}\label{eg!X33}
Working from the bottom up in diagram~\eqref{pic!unproj}, 
let $Y_4\subset\P^4$ be a smooth quartic.
We know $e(Y_4) = -56$ and $h^{2,1}(Y_4) = 30$.
Imposing a linear plane $D=\P^2$ on $Y_4$ gives,
in coordinates $x,y,z,t,u$ of $\P^4$,
\[
\P^2 = D = (x=y=0) \subset \Ybar_4 = (Ax = By)\subset P^4,
\]
where $A,B$ are general cubic forms. Such $\Ybar$ has 9 nodes
at $(A=B=0)\subset D$.
The unprojection of $D\subset Y$ is a quasismooth variety
$X_{3,3}\subset \P(1^5,2)$, which has Fano Hilbert series No.\ 20522.
By~\eqref{eq!e} we have $e(X_{3,3}) = e(Y_4) + 18 - 2 = -40$,
and so $h^{2,1}(X_{3,3}) = 30$.

This calculation is recorded in Table~\ref{tab!codim2}, together
with the numerical data described here.
\end{example}

\subsubsection{10 cases with Type II$_1$ projection}
\label{s!c2y2}

Again we work from bottom up in diagram~\eqref{pic!unproj}.
Thus, for example, to study $X$ whose Hilbert series $P_X$
is no.~6858 in the \grdb\ \cite{grdb}, we observe from that
database (or by hand from the methods of~\cite{ABR02})
that the numerics suggest a Type~II$_1$ projection
to $\Ybar$ with Hilbert series $P_{\Ybar}$ no.~5837,
whose general member we know to be of the form $Y_{10}\subset\P(1,1,2,2,2,3)$.
The task in this case is to impose a divisor $D$ onto a special
(nodal) member of this family, where the divisor $D$ may be singular, but
its normalisation is $\widetilde{D} \cong \P^2$.

\begin{example}
Consider $X = X_{4,6}\subset\P(1,1,2,2,2,3)$, which has
Fano Hilbert series no.~6858 in \cite{grdb}.
As in Example~\ref{eg!X33} we work bottom up,
first constructing $D\subset \Ybar_{10}\subset\P(1,1,2,2,5)$
and then unprojecting. We follow Reid \cite[\S9]{reidgraded} and
Papadakis \cite{papadakis21} for Type II$_1$ unprojections.

In coordinates $x,y,z,t,u$ on $\P(1,1,2,2,5)$, the finite morphism
\begin{eqnarray*}
\P^2\cong\widetilde{D} &\longrightarrow& D\subset\P(1,1,2,2,5) \\
(a,b,c) &\mapsto & (a,b,c^2,(a-b)c,abc^3+c^5)
\end{eqnarray*}
has image $D$ defined by the $2\times 2$ minors of
\[
M =
\begin{pmatrix}
t & u & (x-y)z & (xy+z)z^2 \\
x-y & (xy+z)z & t & u
\end{pmatrix}.
\]
The surface $D$ has two singular points, each of which has
a length~2 preimage in~$\widetilde{D}$: the point
$(1:1:0:0:0)$ is the pinched image of $(1:1:0)\in\widetilde{D}$,
and $(1:1:-1:0:0)$ is the image of two points $(1:1:\pm i)$.

A general $\Ybar_{10}$ containing this $D$ has 34 nodes,
all of which lie on~$D$. (Two lie at the singularities of~$D$,
so the preimage in $\widetilde{D}$ of the singular subscheme of~$\Ybar$ has
length~36 on~$\widetilde{D}$.)

The unprojection of $D\subset \Ybar$ is given by the maximal Pfaffians of the
skew $5\times 5$ matrix
\[
\begin{pmatrix}
 x-y & (xy+z)z & t & u \\
    & s_0 & 1 & s_1 + A_3 \\
        && s_1 & B_6 \\
             &&& zs_0 + C_4
\end{pmatrix}
\quad\text{with entries of degrees}\quad
\begin{pmatrix}
 1 & 4 & 2 & 5\\
    & 2 & 0 & 3 \\ 
         && 3 & 6 \\
            &&& 4
\end{pmatrix}
\]
in $\P(1,1,2,2,5,2,3)$ with coordinates $x,y,z,t,u,s_0,s_1$,
where $A,B,C$ may be determined by the unprojection calculus
if we wish to know them explicitly.
Eliminating $u$ using the linear equation gives $X_{4,6}\subset\P(1,1,2,2,2,3)$,
as required.
We know $e(Y) = -124$, so conclude that $e(X) = -124 + 2\cdot 34 - 2 = -58$ and 
$h^{2,1}(X) = 31$.
\end{example}

\subsubsection{8 cases with no projection}
\label{s!c2noproj}

Our projection techniques do not work in these cases.
We use computer algebra instead.

\begin{example}
Consider a quasismooth Fano 3-fold $X_{6,6}\colon (f=g=0)\subset\P(1,2^3,3^2)$
with Fano Hilbert series number 3508, defined by
\[
f = x^6 + y^3 + z^3 + t^3 + u^2 + v^2
\quad\text{and}\quad
g =  y^2 z + z^2 t + t^2 y + u v.
\]
Iten's Macaulay2 package \cite{ilten} works as follows
(compressing blank lines in the output):
\begin{verbatim}
Macaulay2, version 1.5
with packages: ConwayPolynomials, Elimination, IntegralClosure, LLLBases,
               PrimaryDecomposition, ReesAlgebra, TangentCone
i1 : loadPackage "VersalDeformations"
o1 = VersalDeformations
o1 : Package
i2 : R = QQ[x,y,z,t,u,v,Degrees=>{1,2,2,2,3,3}];
i3 : I = ideal ( x^6 + y^3 + z^3 + t^3 + u^2 + v^2,
         y^2*z + z^2*t + t^2*y + u*v );
o3 : Ideal of R
i4 : CT^1(-1,I)
             2       24
o4 : Matrix R  <--- R
\end{verbatim}
The answer is that $h^{2,1}(X) = \dim T^1_{A_X}(-1) = 24$.

Since $X$ has a K3 elephant $E = (x=0) \subset X$ with basket $9\times\frac{1}2(1,1)$
quotient singularities,
and $h^0(X,T_X)=0$ by Theorem~\ref{th!main}\eqref{th!main2},
we know at this stage from the moduli formula Theorem~\ref{th!main}\eqref{th!main1} that
$h^1(X,T_X)=34$.
This can also be calculated directly by Macaulay2 as follows:
\begin{verbatim}
i5 : CT^1(0,I)
             2       34
o5 : Matrix R  <--- R
\end{verbatim}
Again, the answer is that $h^1(X,T_X) = \dim T^1_{A_X}(0) = 34$.
\end{example}

A similar calculation works with $X_{12,14}\colon (f=g=0)\subset \P(2,3,4,5,6,7)$,
with Hilbert series number 37, with, for example,
\[
f = x^6 + y^4 + z^3 - u^2 + tv
\quad\text{and}\quad
g = x^7 + z^2u + xu^2 + zt^2 + v^2.
\]
In this case there is no elephant $E\subset X$, so the moduli
formula \eqref{eq!moduli} does not apply as stated. However, the Macaulay2 results
are that $h^{2,1}(X) = 18$ and $h^1(X,T_X) = 23$, and so in fact
the formula holds with ``$\alpha_E = 6$'', which is the correct
number calculated on~$X$ from its basket indices and $h^0(X,\cO(1))=0$.

\subsection{Codimension 3}
\label{s!codim3}

There are 70 known deformation families of Fano 3-folds
in codimension~3.
The complete intersection $X=X_{2,2,2}\subset\P^5$ is classical:
the chern class calculation and Lefschetz gives
$e(X)=-24$, $\rho_X=1$ and $h^{2,1}(X)=14$.
The remaining 69 cases are all complete intersections in
weighted Grassmannians $\wGr(2,5)$, and so
$h^{1,1}(X) = 1$ in every case.

\subsubsection{64 cases Type I}
\label{s!c3t1}

We say that a Fano 3-fold $X$ has a {\em Type~I staircase} if
it admits a sequence of alternate Type~I projections and Q-smoothings
to a hypersurface. Concretely, if $X\subset\wP^6$ lies in codimension~3,
then the staircase is
\begin{equation}\label{pic!staircase}
\begin{array}{cccccccc}
 &&&&& \widetilde{Y} & \rightarrow & X \\
 &&&&& \downarrow  \\
 && \widetilde{Y} & \rightarrow & Y \leadsto & \overline{Y} \\
 && \downarrow  \\
Z & \rightsquigarrow & \overline{Z}
\end{array}
\end{equation}
where $X\dashrightarrow \overline{Y}\subset\wP^5$ eliminates a single variable,
$Y\subset\wP^5$ is a general Q-smoothing of $\overline{Y}$,
and $Y \dashrightarrow \overline{Z}$
is a projection to a nodal hypersurface $\overline{Z}\subset\wP^4$ as in
\S\ref{s!codim2}.
Counting nodes on $\overline{Y}$ and $\overline{Z}$ and using the
formula of Theorem~\eqref{th!nodes} completes the calculation of $e(X)$ and $h^{2,1}(X)$.

Of the 64 Fano 3-folds in codimension~3 with a Type~I projection,
57 have a Type I staircase to a hypersurface.

\begin{example}
Consider the family with Hilbert series no.~20523 in \cite{grdb}.
A typical member $X\subset\P(1,1,1,1,1,2,3)$, in coordinates $x_{1\dots5},y,z$,
is given by the five maximal 
Pfaffians of a skew $5\times5$ matrix of forms
\[
\begin{pmatrix} x_1&x_2&A&D\\&x_3&B&E\\&&C&F\\&&&z\end{pmatrix}
\quad\text{where the entries have degrees}\quad
\begin{pmatrix} 1&1&2&2\\&1&2&2\\&&2&2\\&&&3\end{pmatrix}.
\]
It has a quotient singularity $\frac{1}3(1,1,2)$ at the $z$-coordinate point $P_z\in X$.

Projection from that point is calculated by eliminating $z$ from these equations.
Doing that leaves the two Pfaffians of degree~3, which define
\[
\Ybar_{3,3} \colon
\left\{
\begin{pmatrix} A & B & C \\ D & E & F \end{pmatrix}
\begin{pmatrix} x_3\\-x_2\\x_1\end{pmatrix} = \underline{0}\right\}
\subset\P(1,1,1,1,1,2).
\]
For general degree~2 forms $A,\dots,F$,
the image $\Ybar$ has 6 nodes (by Hilbert--Burch) and a Q-smoothing $Y_{3,3}$
which was computed in Example~\ref{eg!X33} above.
Making the projection from $Y_{3,3}$ as in Example~\ref{eg!X33} completes the staircase.
In any case, using the result of Example~\ref{eg!X33} gives
$e(X) = e(Y) + 2\cdot 6 - 2 = -40 + 12 - 2 = -30$, and so $h^{2,1}(X) = 17$.
\end{example}

Of the remaining 7 cases, 4 have a Type~I projection to a family that arises
by Type~II$_1$ unprojection from a hypersurface, so again have a staircase,
but with a more complicated second step. A fifth case has a Type~I projection
to the classical family $Y_{2,3}\subset\P^5$, so also works.

But in two remaining cases, the image of the Type~I projection lies in a family
whose Hodge numbers were computed using the algorithms for $\dim T^1$;
in this paper, these cases remain dependent on computational algebra.

\subsubsection{2 cases Type II$_1$}
\label{s!c3t2}
Of the cases without a Type~I projection, two have a Type~II$_1$ projection:
$X_{7,8,8,9,10}\subset\P(1,2,3,3,4,4,5)$ has a Type II$_1$ projection
from $\frac{1}4(1,1,3)$ and
$X_{10\dots 14}\subset\P(1,3,4,5,5,6,7)$ has a Type II$_1$ projection
from $\frac{1}5(1,2,3)$.
We consider the latter in detail, following Reid \cite[9.5]{reidgraded}
and Papadakis \cite[4.4]{Pap}.

Consider $D \subset \P(1,3,4,5,6)$ defined by the maximal minors of
\[
M_D =
\begin{pmatrix}
t & v & yz & z^2 \\
y & z & t & v
\end{pmatrix}.
\]
This $D$ is the image of $\P(1,2,3)\rightarrow\P(1,3,4,5,6)$
given by $(a,b,c) \mapsto (a,c,b^2,bc,b^3)$; the normalising variable $b$ is recovered
as the ratio of the rows of $M_D$.

The general hypersurface $\Ybar_{18}$ containing $D$ has the form
\[
\Ybar_{18} = (A_{12}m_{12} + B_{11}m_{13}  + 2B_{12}m_{23} + B_{22}m_{24} = 0)
\subset \P(1,3,4,5,6),
\]
where $m_{ij}$ denotes the minor of $M_D$ involving columns $i$ and $j$.

The  unprojection of $D\subset \Ybar_{18}$ is a codimension~3 variety
$X\subset\P(1,3,4,5,5,6,7)$, in coordinates $x,y,z,t,u,v,w$, defined by
the maximal Pfaffians of the skew $5\times 5$ matrix
\[
\begin{pmatrix}
y & z & t & v \\
  & -u & -B_{22} & w + B_{12} \\
     && -w + B_{12} & -B_{11} \\
         &&& -uz - A_{12}
\end{pmatrix}.
\]
For example, setting
\[
A_{12} = yv + y^3 + x^9,\quad
B_{11} = yt + x^8,\quad
B_{12} = 0
\quad\text{and}\quad
B_{22} = v
\]
results in a quasismooth $X$, and $\Ybar_{18}$ whose non-quasismooth locus
is defined by the equations
\begin{gather}
\nonumber 
    z t - y v, \quad
    y^2 z - t^2, \quad 
    y z^2 - t v, \quad
    x^9 y + y^4 + y^2 v + 2 v^2, \quad
    x^9 z - 2 x^8 t - y t^2 + y z v, \\
\nonumber
    z^3 - v^2, \quad
    x^9 t + y^3 t + 2 z^2 v + y t v, \quad
    x^8 y^2 + y^3 t + z^2 v, \quad
    2 x^8 y z - x^9 v + y^3 v - y v^2
\end{gather}
and consists of 22 nodes, all of which lie on $D\subset \Ybar_{18}$.

The general $Y_{18}\subset\P(1,3,4,5,6)$ has $e(Y_{18}) = -80$,
so $e(X) = -38$ and $h^{2,1}(X) = 21$.

\subsubsection{No Type I or II$_1$ projection}
\label{s!c3noproj}

The three remaining cases are
$X_{12\dots 16}\subset\P(1,4,5,5,6,7,8)$,
$X_{16\dots 20}\subset\P(1,5,6,7,8,9,10)$ and
$X_{14\dots 18}\subset\P(1,5,5,6,7,8,9)$.
The first has only a type $\IV$ projection, while the other two do not have
any Gorenstein projections at all.
We compute $T^1$ in these cases:
we work out the first in detail here; the other two are similar.

\begin{example}
A particular $X_{12\dots 16}\subset\P(1,4,5,5,6,7,8)$,
in coordinates $x,y,z,t,u,v,w$, is given by the maximal Pfaffians of the
skew $5\times5$ matrix
\[\begin{pmatrix}
y & z & u & v \\ 
 & u & v &  y^2 + w \\ 
 &  & -y^2 + w & x^9 + yz \\ 
 &  &  & zt + t^2
\end{pmatrix}\] 
in the usual antisymmetric notation. One checks that the scheme defined by those
equations is quasismooth.
We compute $h^{2,1}(X)=20$ and $h^1(X,T_X)=23$ by Macaulay2 as before.

We verify the moduli formula \eqref{th!main1} of Theorem~\ref{th!main}.
The basket of $X$ is
\[
\cB_X= \left\{
\frac{1}{2}(1,1,1),\frac{1}{4}(1,1,3), 2 \times \frac{1}{5}(1,1,4), \frac{1}{5}(1,2,3) \right\}.
\]
The K3 elephant $E=(x=0)\subset X$ is the unique member of $|{-}K_X|$.
It has $h^0(\of_E(1))=0$ and $h^{1,1}(E)=20- \sum r_i-1$, where the $r_i$ are the
indices of singularities of $\cB_X$. Thus
\[
h^1(T_X)-h^{2,1}(X)= \alpha_E- h^{2,2}(X) = (20-1-3-3\cdot 4)-1 = 3,
\]
which agrees with $23-20$.
\end{example}

The other two cases work similarly; in each case $h^{2,1}(X)=20$.

\subsection{Codimension 4}
\label{s!codim4}

All the calculations in codimensions~4 in this section depend on computer algebra:
we use Magma \cite{magma} to compute examples of the codimension~4
equations by unprojection, and Macaulay2 \cite{M2,ilten} for the Hodge numbers.

When a Hilbert series is realised by a Fano 3-fold in codimension~4, it
frequently happens that there is more than one deformation family
of such Fano 3-folds. For 116 of Hilbert series listed in \cite{grdb}
in codimension~4, \cite{TJ} computes the different families, and
observes that they are distinguished by the Euler characteristic
of a quasismooth member. However it does not compute the
Picard rank of these Fano 3-folds, in part because there is no
known format in which they lie as complete intersections, and
so we have no Lefschetz theorem to apply directly.
But the computational methods of this paper still apply, in
conjunction with the unprojection construction of \cite{TJ,papadakis2}.
We compute a few examples here as first calculations.

\begin{example}\label{ex!24097}
{\bf Fano Hilbert series 24097.}
By \cite{TJ} there are 3 families of Fano 3-folds $Y\subset\P(1^6,2^2)$
with (typically) two $\frac{1}2(1,1,1)$ quotient singularities, each with the
Hilbert series No.24097 in \cite{grdb}.
They arise by unprojection of
\[
\P^2 = D \subset \overline{Y} \subset \P(1^6,2),
\]
where $D\subset\P(1^6,2)$ is a linearly embedded plane, and $\overline{Y}$ is defined
by the vanishing of Pfaffians of a skew $5\times 5$ matrix of forms of weights
\begin{equation}\label{wts24097}
\begin{pmatrix} 1&1&1&2\\&1&1&2\\&&1&2\\&&&2 \end{pmatrix}.
\end{equation}
The three famlies arise as so-called ``Tom'' and ``Jerry'' unprojections (see \cite[\S2.3]{TJ}
for details), and the
three different results are listed in the Big Table \cite{BKRbigtable}:
Tom$_1$, Jer$_{12}$ and Jer$_{15}$.
Takagi's analysis \cite[Theorem 0.3]{takagi02} of prime Fano 3-folds
with index~2 terminal singularities shows that the first and third of
these families have $h^{1,1}(X) = 1$.
Using the Macaulay2 computation, and Theorem~\ref{th!main}(\ref{th!main1})
(which holds since each unprojection does indeed carry a quasismooth elephant
$E$ with $\alpha_E = 19 - 1 - 5 = 13$), we complete the table below.
\[
\begin{array}{c|cc|cc|cc}
\text{unproj type} & \#\text{ nodes} & e_X & h^{1,1}(X) & h^{2,1}(X) & h^1(X,T_X) & h^0(X,T_X) \\
\hline
\textrm{Tom}_1 & 6 & -14 & 1 & 9 & 21 & 0 \\
\textrm{Jer}_{12} & 8 & -10 & 3 & 9 & 19 & 0 \\
\textrm{Jer}_{15} & 9 & -12 & 1 & 8 & 20 & 0 \\
\end{array}
\]

For example, the Jer$_{12}$ case above uses $\overline{Y}$ defined by
Pfaffian matrix
\[
\begin{pmatrix} \ \ t \ \  & u \  \ & v & w \\
  & v \ \ & t+u & u x \\
   && x & y^2 - z^2 \\
    &&& y z + t^2 + u^2\ \  \end{pmatrix}
\]
in the coordinates $x,y,z$, $t,u,v$ and $w$ of $\P(1^6,2)$.
Such $\overline{Y}$ contains the plane $D = (t = u = v = w = 0)$.
Unprojecting $D\subset \overline{Y}$ gives $X\subset \P(1^6,2^2)$,
defined by
\begin{gather}
\nonumber    x t - t u - u^2 + v^2, \qquad
     y^2 t - z^2 t - x u^2 + v w, \qquad
   y z t + t^3 + t u^2 - x u v + t w + u w,  \\
\nonumber           y z u + t^2 u + u^3 - y^2 v + z^2 v + x w,  \qquad
    x^2 u - y^2 u + z^2 u - x u^2 + y z v + t^2 v + u^2 v + v w,    \\
\nonumber  x^2 v - x w + t s,  \qquad
   -x y z - x t^2 - x u^2 - x w - u s, \qquad
   -x^3 + x y^2 - x z^2 + x^2 u + v s, \\
 \nonumber   x^2 y^2 - y^4 - x^2 z^2 + 3 y^2 z^2 - z^4 + y z t^2 - x y^2 u + x z^2 u +
        y z u^2 + \\
 \nonumber   \hspace{50mm} + y^2 u v - z^2 u v + x t u v + y z w - x u w - t u w + u^2 w - w s
\end{gather}
in coordinates $x,y,z$, $t,u,v$, $w$ and unprojection variable $s$.
\end{example}

\begin{example}
{\bf Fano Hilbert series 24078.}
By \cite{TJ} there are 3 families of Fano 3-folds $X\subset\P(1^6,2,3)$
with (typically) two $\frac{1}3(1,1,2)$ quotient singularities, each with the
Hilbert series No.24078 in \cite{grdb}.
They arise by unprojection of
\[
\P^2 = D \subset \overline{Y} \subset \P(1^6,2),
\]
where $D\subset\P(1^6,2)$ is a linearly embedded $\P(1,1,2)$, and $\overline{Y}$ is defined
by the vanishing of Pfaffians of a skew $5\times 5$ matrix of forms of the same weights
as \eqref{wts24097} above.

The three different results \cite{BKRbigtable} are:
Tom$_1$, Tom$_5$ and Jer$_{12}$.
In this case the elephant $E\subset X$ has $\alpha_E = 13$,
and the table below summarises the results.
\[
\begin{array}{c|cc|cc|cc}
\text{unproj type} & \#\text{ nodes} & e_X & h^{1,1}(X) & h^{2,1}(X) & h^1(X,T_X) & h^0(X,T_X) \\
\hline
\textrm{Tom}_1 & 5 & -16 & 1 & 10 & 22 & 0 \\
\textrm{Tom}_5 & 4 & -18 & 2 & 12 & 23 & 0 \\
\textrm{Jer}_{12} & 6 & -14 & 1 & 9 & 21 & 0 \\
\end{array}
\]

These calculations seem to be on the limit of what we can do, as
they terminate only when the equations are relatively small.
For example, the Tom$_5$ case above uses $\overline{Y}$ defined by
Pfaffian matrix
\[
\begin{pmatrix} \ \ z \ \  & t \  \ & v+u & w \\
  & u \ \ & t & x v + z u \\
   && z & w - y^2 \\
    &&& x^2 - v^2\ \  \end{pmatrix}
\]
in the coordinates $x,y,z$, $t,u,v$ and $w$ of $\P(1^6,2)$.
\end{example}

Of the 145 Hilbert series of Fano 3-folds listed in \cite{grdb}
as presented naturally in codimension~4, 116 have the numerical properties
consistent with having a Type~I unprojection. The unprojection analysis of these
is the subject of \cite{TJ}, with the results presented in \cite{BKRbigtable},
and in principle they could all be computed as above.
A further 16 Hilbert series have the numerical properties
of a Type~II$_1$ projection, and a computational approach
following Papadakis \cite{papadakis21} is conceivable;
the constructions are part of Taylor's thesis \cite{RT}.

Some of the remaining 13 cases have more complicated projections
that we do not know how to work with systematically yet, but
four cases have no Gorenstein projections at all, and
some other approach is required (even to write down examples
by equations). These cases are:
\[
\begin{array}{lllll}
\text{No.\ 25} & X\subset\P(2,5,6,7,8,9,10,11) &&
\text{No.\ 282} & X\subset\P(1,6,6,7,8,9,10,11)
\\
\text{No.\ 166} & X\subset\P(2,2,3,3,4,4,5,5) 
& \qquad\qquad & 
\text{No.\ 308} & X\subset\P(1,5,6,6,7,8,9,10).
\end{array}
\]

\subsection{A quasismooth unprojection from codimension~4}
\label{s!codim5}

We construct a codimension~4, quasismooth Fano 3-fold
$X\subset \P(1^6,2^2)$ with Hilbert series number 24097
which contains a quasismooth divisor $E\subset X$ that
is itself a complete intersection.
We adapt Example~\ref{ex!24097} so that the codimension~3
projection $Y\subset \P(1^6,2)$ contains two divisors:
the coordinate planes $D=\P^2$ and $E=\P(1,1,2)$
meeting along the coordinate line $\P^1$.

Indeed define $Y$ by the maximal Pfaffians of
\[
\begin{pmatrix} \ \ t \ \  & u \  \ & v & w \\
  & v \ \ & u & -z v-u^2 \\
   && z-t & y z - x^2 \\
    &&& y^2 - t^2\ \  \end{pmatrix}
\]
in the coordinates $x,y,z$, $t,u,v$ and $w$ of $\P(1^6,2)$.
Then $D = (t=u=v=w=0) = \P^2$ lies inside $Y$ in Jer$_{12}$ format
while $E=(z=t=u=v=0) = \P(1,1,2)$ lies inside $Y$ in Tom$_5$ format.

Altogether $Y$ has $8$ nodes; these all lie on $D$
(in accordance with Jer$_{12}$ unprojection of $D$ to construct 
Hilbert series 24097), and 4 of them lie on the intersection $D\cap E$
(in accordance with the Tom$_5$ unprojection or $E$ to
construct Hilbert series 24078).

We may unproject either divisor, and we choose to
unproject $D\subset Y$ to give $X\subset\P(1^6,2^2)$.
All the 8 nodes are resolved by this, and $X$ is quasismooth.
The Fano 3-fold $X$ has Picard rank $\rho_X = 3$
(as in Example~\ref{ex!24097} above).

Furthermore, $E\subset Y$ has birational image in $X$,
which we also denote $E\subset X$ defined by equations
\[
E = (z=t=u=v=0) \cap X \subset \P(1^6,2^2),
\]
in coordinates $x,y,z,t,u,v,w,s$.
Computing the unprojection shows that
$E \cong (x^4 - y^4 - w^2 + w s = 0)\subset \P(1^2,2^2)$ in coordinates $x,y,w,s$,
which is $\P(1,1,2)$ blown up in 4 points on the coordinate line $L = \P(1,1)$
followed by the contraction of the resulting $-2$-curve $\widetilde{L}$,
the birational transform of $L$.
Thus it is a index~2 Fano surface with two $\frac{1}2(1,1)$ quotient singularities,
Picard rank~4 and $K_E^2 = 4$. It can be unprojected to an ordinary, isolated cDV
singular point (in new local coordinates, the cone on $E$)
on an otherwise smooth complete intersection $Z_{2,2,2}\subset\P^6$.

\bibliographystyle{alpha}
\bibliography{bibliography}

\appendix
\section{Hodge numbers of Fano 3-folds}
\label{app!fano}

Tables~\ref{tab!codim1}--\ref{tab!codim3} in \ref{s!tables} below list the invariants
for all known families of Fano 3-folds in codimension at most~3.
The majority of the calculations can be carried out by hand.
We use computer algebra where not, and also use it as a
check on all results.

In codimensions~1 and~2 respectively the Fano 3-folds come from
Iano-Fletcher (\cite{fletcher} Tables~5 and~6 respectively;
in codimensions~3 and~4 they are from Alt{\i}nok (\cite{Aphd}).
The graded ring database identifier (denoted `\grdb' in the tables) is that of~\cite{grdb}.

\subsection{Our use of computer algebra}
\label{s!comp}

The explicit calculations we need are standard, although sometimes rather involved.
There are three places computer algebra may assist.
\begin{enumerate}
\item
Checking that a variety is quasismooth can usually be done
with Bertini's theorem. In codimension~3 and~4, this can be
carried out as in \cite[\S3--4]{BG}, for example, when Type~I
projections (and staircases) are available.
In other cases, we check the Jacobian condition by machine.
This, or some equivalent (such as \cite[Theorem~5.5]{tonoli} or \cite{BFK}),
can be checked by computer algebra given explicit equations.
\item
Checking that a variety has only ordinary nodes as singularities,
and counting those nodes, can again usually be done by
Bertini's theorem together with a chern class calculation
when we have Type~I projections;
see for example \cite[\S4]{BG} for the nodes
and \cite[\S7]{TJ} for the count.
In other cases, we use computer algebra following \cite[\S6]{TJ}.
\item
Computing the dimensions of graded pieces of spaces $T^1_{A_X}$
seems too hard by hand in most cases, but there are algorithms
to do this based on Gr{\"o}bner basis.
\end{enumerate}

We are indebted to the developers of the computer algebra systems
Macaulay2 \cite{M2}, Magma \cite{magma} and Singular \cite{singular}
that we used for these calculations, and to Ilten \cite{ilten}
for the Versal Deformation package for Macaulay2.
(The latter conveniently handles the gradings on variables automatically when
computing graded pieces of $T^1_{A_X}$; 
on other systems we had to pick out the graded piece given generators for
the whole module ``by hand''.)

In practice, most computations here work when the equations of
the Fano 3-fold are fairly sparse, and as the codimension increases
it becomes harder to find such sparse representatives.

\subsection{Blache's orbifold formula}
\label{s!blache}

Let $V$ be a projective orbifold of dimension~$n$,
embedded as a quasismooth subvariety
of weighted projective space $V\subset \P=\P(a_0,\dots,a_N)$.
We suppose, in addition, that $V$ is a manifold away
from a finite set of strictly orbifold points $Q_1,\dots, Q_s\in V$.

We define the orbifold total chern class $\corb(T_\P) = 1 + c_{1,\mathrm{orb}}(T_P) + \cdots + \corbn(T_\P)$ of $\P$ via
\[
0 \rightarrow \cO_\P \rightarrow \oplus_{i=0}^N \cO_\P(a_i) \rightarrow T_\P \rightarrow 0.
\]
Taking the restriction of this to $V$, we derive the top chern class $\corb(V)$ of $V$
from the tangent exact sequence
\[
0 \rightarrow T_V \rightarrow T_{\P|V} \rightarrow N_{V|\P} \rightarrow 0
\]
exactly as in the smooth case: that is, we make the formal computation
\[
1 + c_{1,\mathrm{orb}}(T_\P) + \cdots + {\corb}_{,N}(T_\P) = \corb(T_\P) := \prod (1 + a_i h),
\]
where $H^2(\P,\Q) = h\Q$ and ${\corb}_{,j}\in H^{2j}(\P,\Q)$, and then
\[
\left(1+ {\corb}_{,1}(T_V) + \cdots + {\corb}_{,n}(T_V)\right)c(N_{V|\P}) =  \corb(T_\P).
\]

Then we define the orbifold euler class $\eorb(V)$ by
\[
\eorb(V) := \int_V \corbn(V) \in\Q.
\]
This is a formal computation that ignores orbifold behaviour.
However, it is related to the topological euler characteristic $e(V)$
by the following theorem of Blache \cite{blache}.

\begin{thm}[\cite{blache} (2.11--14)]
Let $V$ be a projective orbifold with finite orbifold locus as above.
Then $\eorb(X)\in\Q$ satisfies
\[
e(X) = \eorb(X) + \sum_{Q\in\cB} \frac{r-1}r,
\]
where $r = r(Q)$ is the local index of the orbifold point $Q$.
\end{thm}

For a hypersurface $X_d \subset \P(a_0,\dots,a_{n+1})$ we have
\[
\eorb(X) = \text{the coefficient of $h^n$ in series expansion of }\left(\frac{\prod (1 + a_i h)}{1+dh} \deg(X)\right).
\]

For example, Fano number 337 is $X_{28} \subset\P(1,4,6,7,11)$
and has basket
\[
\cB = \left\{ 2\times\frac{1}2(1,1,1), \frac{1}6(1,1,5), \frac{1}{11}(1,4,7)\right\}.
\]
Calculating as above gives
\begin{eqnarray*}
e(X) &=& \eorb(X) + 2 \times \frac{1}2 + \frac{5}6 + \frac{10}{11} \\
 & = & \text{coeff}_{h^3} \left((1 + 29h + 309h^2) (1 - 28h + 784h^2 - 21952h^3)\right) \frac{28}{4\cdot 6 \cdot 7 \cdot 11} \\
 && \qquad\qquad + 2 \times \frac{1}2 + \frac{5}6 + \frac{10}{11} \\
&=& \text{coeff}_{h^3}(1 + h + 281h^2 - 6385h^3)\frac{1}{66} + 2 \times \frac{1}2 + \frac{5}6 + \frac{10}{11} \\
&=&\frac{-6385}{66} + 1 + 5/6 + 10/11 \\
&=&-94.
\end{eqnarray*}
This agrees with our calculation $h^{2,1}(X) = 49$ and $e(X) = 4 - 2\times49$.

\subsection{Tables of results}
\label{s!tables}

Tables~\ref{tab!codim1}--\ref{tab!codim3} list the Hodge number $h^{2,1}(X)$,
the topological euler characteristic $e(X)$ and the number of moduli
$h^1(T_X) = \dim H^1(X,T_X)$
for quasismooth members $X$ of the families of Fano 3-folds
in codimensions~1--3 respectively.

In codimension~1, we apply the Griffith's Residue Theorem in \S\ref{s!intro1.1}
together with the formulas of Theorem~\ref{th!main}.
In codimension~2,
Table~\ref{tab!codim2} documents the method we use to compute
the invariants. This could be the conventional chern class calculation,
indicated by $c_3(T_X)$, a computer calculation of $T^1_{A_X}$,
indicated by $T^1$, or a projection calculation, indicated by $\I$ or II$_1$
depending on the type of the projection.
Where we use a projection, we also list the centre $\frac{1}r$ of projection
(leaving the polarising weights of $\frac{1}r(1,a,-a)$ implicit),
the number of nodes on the image of projection, and the
number of that image in the \grdb.
Where there is more than one possible centre of projection,
we list them all.
Combining this data with the results of Table~\ref{tab!codim1}
and Theorems~\ref{th!main} and~\ref{th!nodes}
calculates the invariants. For example, number~25022, $X_{3,3}\subset\P(1^5,2)$ (the second line
in Table~\ref{tab!codim2}) projects to number~20521 with 9 nodes; the Euler charactistic of the smoothed image is listed in Table~\ref{tab!codim1} as $-56$, and so the for $X_{3,3}$ it is $-56 + 2\times 9 - 2 = -40$, as displayed.

In codimension~3, Table~\ref{tab!codim3} documents the method we
use in the 70 cases as follows:
\begin{enumerate}
\item
57 cases have at least one `staircase' of 
two Type~I projections to a hypersurface.
This is indicated by $\I$--$\I$.
\item
4 cases have a Type~I projection to a codimension~2
family that has as a Type~II$_1$ projection to a hypersurface (indicated by $\I$--II$_1$).
\item
2 cases have a Type~II$_1$ projection directly to a hypersurface (II$_1$).
\item
2 cases have a Type~I projection to a codimension~2 family with
no projection ($\I$--$T^1$).
\item
1 case has a Type~I projection to a known smooth Fano (I--smooth).
\item
1 case is a known smooth Fano complete intersection ($c_3(T_X)$).
\item
3 cases have no Type~I or II$_1$ projections at all ($T^1$).
\end{enumerate}
Again, where there is a projection from $X$ we list the centre $\frac{1}r$,
the number of nodes and the \grdb\ identifier for each possibility, and
applying Theorems~\ref{th!main} and~\ref{th!nodes} together with
data from previous tables calculates the invariants.

\begin{longtable}{clrrr}
\caption{Codimension 1: $h^{1,1}(X)=1$ and $h^0(X,T_X)=0$ in all cases.}\label{tab!codim1}\\
\toprule
\multicolumn{1}{c}{\grdb}&\multicolumn{1}{c}{variety}&\multicolumn{1}{c}{$h^{2,1}$}&\multicolumn{1}{c}{$e(X)$}&\multicolumn{1}{c}{$h^1(T_X)$}\\
\cmidrule(lr){1-1}\cmidrule(lr){2-2}\cmidrule(lr){3-4}\cmidrule(lr){5-5}
\endfirsthead
\multicolumn{5}{l}{\vspace{-0.25em}\scriptsize\emph{\tablename\ \thetable{} continued from previous page}}\\
\midrule
\endhead
\multicolumn{5}{r}{\scriptsize\emph{Continued on next page}}\\
\endfoot
\bottomrule
\endlastfoot
\evnrow 20521&$X_{4}\subset\PP^4$&$30$&$-56$&$43$\\
16203&$X_{5}\subset\PP(1,1,1,1,2)$&$38$&$-72$&$51$\\
\evnrow 16202&$X_{6}\subset\PP(1,1,1,1,3)$&$52$&$-100$&$66$\\
11101&$X_{6}\subset\PP(1,1,1,2,2)$&$41$&$-78$&$55$\\
\evnrow 10981&$X_{7}\subset\PP(1,1,1,2,3)$&$51$&$-98$&$63$\\
10980&$X_{8}\subset\PP(1,1,1,2,4)$&$64$&$-124$&$78$\\
\evnrow 10960&$X_{9}\subset\PP(1,1,1,3,4)$&$71$&$-138$&$83$\\
10959&$X_{10}\subset\PP(1,1,1,3,5)$&$85$&$-166$&$98$\\
\evnrow 10958&$X_{12}\subset\PP(1,1,1,4,6)$&$111$&$-218$&$125$\\
5838&$X_{8}\subset\PP(1,1,2,2,3)$&$45$&$-86$&$58$\\
\evnrow 5837&$X_{10}\subset\PP(1,1,2,2,5)$&$64$&$-124$&$79$\\
5257&$X_{9}\subset\PP(1,1,2,3,3)$&$49$&$-94$&$62$\\
\evnrow 5157&$X_{10}\subset\PP(1,1,2,3,4)$&$56$&$-108$&$66$\\
5153&$X_{11}\subset\PP(1,1,2,3,5)$&$65$&$-126$&$74$\\
\evnrow 5152&$X_{12}\subset\PP(1,1,2,3,6)$&$75$&$-146$&$88$\\
5137&$X_{12}\subset\PP(1,1,2,4,5)$&$70$&$-136$&$81$\\
\evnrow 5136&$X_{14}\subset\PP(1,1,2,4,7)$&$90$&$-176$&$102$\\
5134&$X_{15}\subset\PP(1,1,2,5,7)$&$97$&$-190$&$106$\\
\evnrow 5133&$X_{16}\subset\PP(1,1,2,5,8)$&$108$&$-212$&$119$\\
5132&$X_{18}\subset\PP(1,1,2,6,9)$&$128$&$-252$&$141$\\
\evnrow 4984&$X_{12}\subset\PP(1,1,3,4,4)$&$60$&$-116$&$73$\\
4909&$X_{13}\subset\PP(1,1,3,4,5)$&$66$&$-128$&$73$\\
\evnrow 4907&$X_{15}\subset\PP(1,1,3,4,7)$&$82$&$-160$&$89$\\
4906&$X_{16}\subset\PP(1,1,3,4,8)$&$91$&$-178$&$102$\\
\evnrow 4893&$X_{15}\subset\PP(1,1,3,5,6)$&$78$&$-152$&$87$\\
4892&$X_{18}\subset\PP(1,1,3,5,9)$&$104$&$-204$&$114$\\
\evnrow 4891&$X_{21}\subset\PP(1,1,3,7,10)$&$126$&$-248$&$133$\\
4890&$X_{22}\subset\PP(1,1,3,7,11)$&$136$&$-268$&$144$\\
\evnrow 4889&$X_{24}\subset\PP(1,1,3,8,12)$&$154$&$-304$&$165$\\
4835&$X_{16}\subset\PP(1,1,4,5,6)$&$77$&$-150$&$83$\\
\evnrow 4834&$X_{20}\subset\PP(1,1,4,5,10)$&$108$&$-212$&$119$\\
4822&$X_{18}\subset\PP(1,1,4,6,7)$&$88$&$-172$&$94$\\
\evnrow 4821&$X_{22}\subset\PP(1,1,4,6,11)$&$120$&$-236$&$127$\\
4820&$X_{28}\subset\PP(1,1,4,9,14)$&$165$&$-326$&$172$\\
\evnrow 4819&$X_{30}\subset\PP(1,1,4,10,15)$&$182$&$-360$&$190$\\
4807&$X_{21}\subset\PP(1,1,5,7,8)$&$99$&$-194$&$104$\\
\evnrow 4806&$X_{26}\subset\PP(1,1,5,7,13)$&$137$&$-270$&$143$\\
4805&$X_{36}\subset\PP(1,1,5,12,18)$&$211$&$-418$&$218$\\
\evnrow 4794&$X_{24}\subset\PP(1,1,6,8,9)$&$110$&$-216$&$115$\\
4793&$X_{30}\subset\PP(1,1,6,8,15)$&$154$&$-304$&$160$\\
\evnrow 4792&$X_{42}\subset\PP(1,1,6,14,21)$&$240$&$-476$&$247$\\
2402&$X_{12}\subset\PP(1,2,2,3,5)$&$47$&$-90$&$59$\\
\evnrow 2401&$X_{14}\subset\PP(1,2,2,3,7)$&$60$&$-116$&$74$\\
1389&$X_{12}\subset\PP(1,2,3,3,4)$&$40$&$-76$&$54$\\
\evnrow 1162&$X_{14}\subset\PP(1,2,3,4,5)$&$45$&$-86$&$52$\\
1160&$X_{16}\subset\PP(1,2,3,4,7)$&$54$&$-104$&$62$\\
\evnrow 1159&$X_{18}\subset\PP(1,2,3,4,9)$&$65$&$-126$&$76$\\
1155&$X_{15}\subset\PP(1,2,3,5,5)$&$48$&$-92$&$60$\\
\evnrow 1149&$X_{17}\subset\PP(1,2,3,5,7)$&$56$&$-108$&$60$\\
1147&$X_{18}\subset\PP(1,2,3,5,8)$&$61$&$-118$&$66$\\
\evnrow 1146&$X_{20}\subset\PP(1,2,3,5,10)$&$72$&$-140$&$82$\\
1144&$X_{21}\subset\PP(1,2,3,7,9)$&$72$&$-140$&$78$\\
\evnrow 1143&$X_{24}\subset\PP(1,2,3,7,12)$&$89$&$-174$&$97$\\
1142&$X_{24}\subset\PP(1,2,3,8,11)$&$87$&$-170$&$93$\\
\evnrow 1141&$X_{26}\subset\PP(1,2,3,8,13)$&$99$&$-194$&$106$\\
1140&$X_{30}\subset\PP(1,2,3,10,15)$&$121$&$-238$&$131$\\
\evnrow 1113&$X_{20}\subset\PP(1,2,4,5,9)$&$62$&$-120$&$70$\\
1112&$X_{22}\subset\PP(1,2,4,5,11)$&$72$&$-140$&$81$\\
\evnrow 1079&$X_{20}\subset\PP(1,2,5,6,7)$&$55$&$-106$&$60$\\
1078&$X_{26}\subset\PP(1,2,5,6,13)$&$80$&$-156$&$87$\\
\evnrow 1076&$X_{27}\subset\PP(1,2,5,9,11)$&$77$&$-150$&$79$\\
1075&$X_{32}\subset\PP(1,2,5,9,16)$&$100$&$-196$&$104$\\
\evnrow 1074&$X_{42}\subset\PP(1,2,5,14,21)$&$144$&$-284$&$150$\\
1067&$X_{30}\subset\PP(1,2,6,7,15)$&$88$&$-172$&$96$\\
\evnrow 866&$X_{15}\subset\PP(1,3,3,4,5)$&$40$&$-76$&$52$\\
545&$X_{18}\subset\PP(1,3,4,5,6)$&$42$&$-80$&$49$\\
\evnrow 539&$X_{19}\subset\PP(1,3,4,5,7)$&$45$&$-86$&$47$\\
537&$X_{20}\subset\PP(1,3,4,5,8)$&$48$&$-92$&$53$\\
\evnrow 536&$X_{24}\subset\PP(1,3,4,5,12)$&$63$&$-122$&$71$\\
534&$X_{24}\subset\PP(1,3,4,7,10)$&$57$&$-110$&$58$\\
\evnrow 533&$X_{28}\subset\PP(1,3,4,7,14)$&$72$&$-140$&$80$\\
532&$X_{30}\subset\PP(1,3,4,10,13)$&$74$&$-144$&$75$\\
\evnrow 531&$X_{34}\subset\PP(1,3,4,10,17)$&$90$&$-176$&$92$\\
530&$X_{36}\subset\PP(1,3,4,11,18)$&$97$&$-190$&$101$\\
\evnrow 529&$X_{42}\subset\PP(1,3,4,14,21)$&$120$&$-236$&$125$\\
508&$X_{21}\subset\PP(1,3,5,6,7)$&$45$&$-86$&$51$\\
\evnrow 507&$X_{33}\subset\PP(1,3,5,11,14)$&$74$&$-144$&$74$\\
506&$X_{38}\subset\PP(1,3,5,11,19)$&$92$&$-180$&$93$\\
\evnrow 505&$X_{48}\subset\PP(1,3,5,16,24)$&$126$&$-248$&$130$\\
500&$X_{24}\subset\PP(1,3,6,7,8)$&$48$&$-92$&$56$\\
\evnrow 356&$X_{24}\subset\PP(1,4,5,6,9)$&$45$&$-86$&$47$\\
355&$X_{30}\subset\PP(1,4,5,6,15)$&$62$&$-120$&$69$\\
\evnrow 353&$X_{25}\subset\PP(1,4,5,7,9)$&$46$&$-88$&$46$\\
352&$X_{32}\subset\PP(1,4,5,7,16)$&$65$&$-126$&$69$\\
\evnrow 351&$X_{44}\subset\PP(1,4,5,13,22)$&$91$&$-178$&$91$\\
350&$X_{54}\subset\PP(1,4,5,18,27)$&$120$&$-236$&$121$\\
\evnrow 337&$X_{28}\subset\PP(1,4,6,7,11)$&$49$&$-94$&$50$\\
336&$X_{34}\subset\PP(1,4,6,7,17)$&$65$&$-126$&$67$\\
\evnrow 296&$X_{27}\subset\PP(1,5,6,7,9)$&$42$&$-80$&$42$\\
295&$X_{30}\subset\PP(1,5,6,8,11)$&$46$&$-88$&$45$\\
\evnrow 294&$X_{38}\subset\PP(1,5,6,8,19)$&$64$&$-124$&$64$\\
293&$X_{66}\subset\PP(1,5,6,22,33)$&$120$&$-236$&$120$\\
\evnrow 289&$X_{40}\subset\PP(1,5,7,8,20)$&$64$&$-124$&$68$\\
271&$X_{36}\subset\PP(1,7,8,9,12)$&$42$&$-80$&$41$\\
\evnrow 270&$X_{50}\subset\PP(1,7,8,10,25)$&$63$&$-122$&$62$\\
\end{longtable}

{\ }

\begin{longtable}{clcccrc}
\caption{Codimension 2: $h^{1,1}(X)=1$ and $h^0(X,T_X)=0$ in all cases.}\label{tab!codim2}\\
\toprule
\multicolumn{1}{c}{grdb}&\multicolumn{1}{c}{variety}& \multicolumn{1}{c}{method} & \multicolumn{1}{c}{$\frac{1}r$, \#nodes, target id} & \multicolumn{1}{c}{$h^{2,1}$}&\multicolumn{1}{c}{$e(X)$}&\multicolumn{1}{c}{$h^1(T_X)$}\\
\cmidrule(lr){1-1}\cmidrule(lr){2-2}\cmidrule(lr){3-4}\cmidrule(lr){5-7}
\endfirsthead
\multicolumn{5}{l}{\vspace{-0.25em}\scriptsize\emph{\tablename\ \thetable{} continued from previous page}}\\
\midrule
\endhead
\multicolumn{5}{r}{\scriptsize\emph{Continued on next page}}\\
\endfoot
\bottomrule
\endlastfoot
\evnrow 24076&$X_{2,3}\subset\PP^5$&$c_3(T_X)$&&$20$&$-36$&$34$\\
20522&$X_{3,3}\subset\PP(1,1,1,1,1,2)$&I&$\frac{1}{2},9,20521$&$22$&$-40$&$36$\\
\evnrow 16225&$X_{3,4}\subset\PP(1,1,1,1,2,2)$&I&$\frac{1}{2},12,16203$&$27$&$-50$&$41$\\
16204&$X_{4,4}\subset\PP(1,1,1,1,2,3)$&I&$\frac{1}{3},8,16203$&$31$&$-58$&$45$\\
\evnrow 11435&$X_{4,4}\subset\PP(1,1,1,2,2,2)$&I&$\frac{1}{2},16,11101$&$26$&$-48$&$39$\\
11102&$X_{4,5}\subset\PP(1,1,1,2,2,3)$&I&$\frac{1}{2},20,10981$; $\frac{1}{3},10,11101$&$32$&$-60$&$45$\\
\evnrow 11002&$X_{4,6}\subset\PP(1,1,1,2,3,3)$&I&$\frac{1}{3},12,10981$&$40$&$-76$&$53$\\
10983&$X_{5,6}\subset\PP(1,1,1,2,3,4)$&I&$\frac{1}{2},30,10960$; $\frac{1}{4},10,10981$&$42$&$-80$&$55$\\
\evnrow 10982&$X_{6,6}\subset\PP(1,1,1,2,3,5)$&I&$\frac{1}{5},6,10981$&$46$&$-88$&$59$\\
10961&$X_{6,8}\subset\PP(1,1,1,3,4,5)$&I&$\frac{1}{5},12,10960$&$60$&$-116$&$73$\\
\evnrow 6858&$X_{4,6}\subset\PP(1,1,2,2,2,3)$&II$_1$&$\frac{1}{2},34,5837$&$31$&$-58$&$43$\\
5857&$X_{5,6}\subset\PP(1,1,2,2,3,3)$&I&$\frac{1}{3},15,5838$&$31$&$-58$&$42$\\
\evnrow 5843&$X_{6,6}\subset\PP(1,1,2,2,3,4)$&I&$\frac{1}{4},12,5838$&$34$&$-64$&$45$\\
5839&$X_{6,7}\subset\PP(1,1,2,2,3,5)$&I&$\frac{1}{5},7,5838$&$39$&$-74$&$50$\\
\evnrow 5514&$X_{6,6}\subset\PP(1,1,2,3,3,3)$&I&$\frac{1}{3},18,5257$&$32$&$-60$&$42$\\
5261&$X_{6,7}\subset\PP(1,1,2,3,3,4)$&I&$\frac{1}{3},21,5157$; $\frac{1}{4},14,5257$&$36$&$-68$&$46$\\
\evnrow 5258&$X_{6,8}\subset\PP(1,1,2,3,3,5)$&I&$\frac{1}{3},24,5153$; $\frac{1}{5},8,5257$&$42$&$-80$&$52$\\
5200&$X_{6,8}\subset\PP(1,1,2,3,4,4)$&I&$\frac{1}{4},16,5157$&$41$&$-78$&$51$\\
\evnrow 5161&$X_{7,8}\subset\PP(1,1,2,3,4,5)$&I&$\frac{1}{3},28,5137$; $\frac{1}{5},14,5157$&$43$&$-82$&$53$\\
5159&$X_{6,9}\subset\PP(1,1,2,3,4,5)$&I&$\frac{1}{4},18,5153$; $\frac{1}{5},9,5157$&$48$&$-92$&$58$\\
\evnrow 5158&$X_{8,9}\subset\PP(1,1,2,3,4,7)$&I&$\frac{1}{7},6,5157$&$51$&$-98$&$61$\\
5156&$X_{6,10}\subset\PP(1,1,2,3,5,5)$&I&$\frac{1}{5},10,5153$&$56$&$-108$&$66$\\
\evnrow 5155&$X_{8,10}\subset\PP(1,1,2,3,5,7)$&I&$\frac{1}{3},40,5134$; $\frac{1}{7},8,5153$&$58$&$-112$&$68$\\
5154&$X_{9,10}\subset\PP(1,1,2,3,5,8)$&I&$\frac{1}{8},6,5153$&$60$&$-116$&$70$\\
\evnrow 5138&$X_{8,10}\subset\PP(1,1,2,4,5,6)$&I&$\frac{1}{6},16,5137$&$55$&$-106$&$65$\\
5135&$X_{10,14}\subset\PP(1,1,2,5,7,9)$&I&$\frac{1}{9},10,5134$&$88$&$-172$&$98$\\
\evnrow 4985&$X_{8,9}\subset\PP(1,1,3,4,4,5)$&I&$\frac{1}{4},24,4909$; $\frac{1}{5},18,4984$&$43$&$-82$&$51$\\
4936&$X_{8,10}\subset\PP(1,1,3,4,5,5)$&I&$\frac{1}{5},20,4909$&$47$&$-90$&$55$\\
\evnrow 4912&$X_{9,10}\subset\PP(1,1,3,4,5,6)$&I&$\frac{1}{4},30,4893$; $\frac{1}{6},18,4909$&$49$&$-94$&$57$\\
4911&$X_{8,12}\subset\PP(1,1,3,4,5,7)$&I&$\frac{1}{5},24,4907$; $\frac{1}{7},8,4909$&$59$&$-114$&$67$\\
\evnrow 4910&$X_{10,12}\subset\PP(1,1,3,4,5,9)$&I&$\frac{1}{9},6,4909$&$61$&$-118$&$69$\\
4908&$X_{12,14}\subset\PP(1,1,3,4,7,11)$&I&$\frac{1}{11},6,4907$&$77$&$-150$&$85$\\
\evnrow 4894&$X_{10,12}\subset\PP(1,1,3,5,6,7)$&I&$\frac{1}{7},20,4893$&$59$&$-114$&$67$\\
4848&$X_{10,12}\subset\PP(1,1,4,5,6,6)$&I&$\frac{1}{6},24,4835$&$54$&$-104$&$61$\\
\evnrow 4837&$X_{11,12}\subset\PP(1,1,4,5,6,7)$&I&$\frac{1}{5},33,4822$; $\frac{1}{7},22,4835$&$56$&$-108$&$63$\\
4836&$X_{12,15}\subset\PP(1,1,4,5,6,11)$&I&$\frac{1}{11},6,4835$&$72$&$-140$&$79$\\
\evnrow 4823&$X_{12,14}\subset\PP(1,1,4,6,7,8)$&I&$\frac{1}{8},24,4822$&$65$&$-126$&$72$\\
4808&$X_{14,16}\subset\PP(1,1,5,7,8,9)$&I&$\frac{1}{9},28,4807$&$72$&$-140$&$78$\\
\evnrow 4795&$X_{16,18}\subset\PP(1,1,6,8,9,10)$&I&$\frac{1}{10},32,4794$&$79$&$-154$&$85$\\
3508&$X_{6,6}\subset\PP(1,2,2,2,3,3)$&$T^1$&&$24$&$-44$&$34$\\
\evnrow 2419&$X_{6,8}\subset\PP(1,2,2,3,3,4)$&II$_1$&$\frac{1}{3},33,2401$&$28$&$-52$&$37$\\
2409&$X_{6,10}\subset\PP(1,2,2,3,4,5)$&II$_1$&$\frac{1}{4},25,2401$&$36$&$-68$&$45$\\
\evnrow 2403&$X_{9,10}\subset\PP(1,2,2,3,5,7)$&I&$\frac{1}{7},9,2402$&$39$&$-74$&$47$\\
1390&$X_{8,9}\subset\PP(1,2,3,3,4,5)$&I&$\frac{1}{5},12,1389$&$29$&$-54$&$36$\\
\evnrow 1249&$X_{8,10}\subset\PP(1,2,3,4,4,5)$&II$_1$&$\frac{1}{4},36,1159$&$30$&$-56$&$37$\\
1179&$X_{9,10}\subset\PP(1,2,3,4,5,5)$&I&$\frac{1}{5},15,1162$&$31$&$-58$&$37$\\
\evnrow 1171&$X_{8,12}\subset\PP(1,2,3,4,5,6)$&II$_1$&$\frac{1}{5},30,1159$&$36$&$-68$&$43$\\
1165&$X_{10,11}\subset\PP(1,2,3,4,5,7)$&I&$\frac{1}{7},11,1162$&$35$&$-66$&$41$\\
\evnrow 1164&$X_{9,12}\subset\PP(1,2,3,4,5,7)$&I&$\frac{1}{5},18,1160$; $\frac{1}{7},9,1162$&$37$&$-70$&$43$\\
1163&$X_{10,12}\subset\PP(1,2,3,4,5,8)$&I&$\frac{1}{8},8,1162$&$38$&$-72$&$44$\\
\evnrow 1161&$X_{12,14}\subset\PP(1,2,3,4,7,10)$&I&$\frac{1}{10},8,1160$&$47$&$-90$&$53$\\
1156&$X_{10,12}\subset\PP(1,2,3,5,5,7)$&I&$\frac{1}{5},20,1149$; $\frac{1}{7},12,1155$&$37$&$-70$&$42$\\
\evnrow 1154&$X_{10,14}\subset\PP(1,2,3,5,7,7)$&I&$\frac{1}{7},14,1149$&$43$&$-82$&$48$\\
1152&$X_{10,15}\subset\PP(1,2,3,5,7,8)$&I&$\frac{1}{7},15,1147$; $\frac{1}{8},10,1149$&$47$&$-90$&$52$\\
\evnrow 1151&$X_{12,14}\subset\PP(1,2,3,5,7,9)$&I&$\frac{1}{5},28,1144$; $\frac{1}{9},12,1149$&$45$&$-86$&$50$\\
1150&$X_{14,15}\subset\PP(1,2,3,5,7,12)$&I&$\frac{1}{12},6,1149$&$51$&$-98$&$56$\\
\evnrow 1148&$X_{15,16}\subset\PP(1,2,3,5,8,13)$&I&$\frac{1}{13},6,1147$&$56$&$-108$&$61$\\
1145&$X_{14,18}\subset\PP(1,2,3,7,9,11)$&I&$\frac{1}{11},14,1144$&$59$&$-114$&$64$\\
\evnrow 1121&$X_{10,12}\subset\PP(1,2,4,5,5,6)$&II$_1$&$\frac{1}{5},40,1112$&$33$&$-62$&$39$\\
1114&$X_{10,14}\subset\PP(1,2,4,5,6,7)$&II$_1$&$\frac{1}{6},35,1112$&$38$&$-72$&$44$\\
\evnrow 1083&$X_{12,16}\subset\PP(1,2,5,6,7,8)$&II$_1$&$\frac{1}{5},48,1067$; $\frac{1}{7},40,1078$&$41$&$-78$&$46$\\
1080&$X_{14,15}\subset\PP(1,2,5,6,7,9)$&I&$\frac{1}{9},15,1079$&$41$&$-78$&$45$\\
\evnrow 1077&$X_{18,22}\subset\PP(1,2,5,9,11,13)$&I&$\frac{1}{13},18,1076$&$60$&$-116$&$63$\\
1068&$X_{14,18}\subset\PP(1,2,6,7,8,9)$&II$_1$&$\frac{1}{8},45,1067$&$44$&$-84$&$49$\\
\evnrow 867&$X_{10,12}\subset\PP(1,3,3,4,5,7)$&I&$\frac{1}{7},10,866$&$31$&$-58$&$36$\\
640&$X_{10,12}\subset\PP(1,3,4,4,5,6)$&$T^1$&&$28$&$-52$&$33$\\
\evnrow 547&$X_{12,13}\subset\PP(1,3,4,5,6,7)$&I&$\frac{1}{7},13,545$&$30$&$-56$&$34$\\
546&$X_{12,15}\subset\PP(1,3,4,5,6,9)$&I&$\frac{1}{9},9,545$&$34$&$-64$&$38$\\
\evnrow 544&$X_{12,14}\subset\PP(1,3,4,5,7,7)$&I&$\frac{1}{7},14,539$&$32$&$-60$&$35$\\
542&$X_{12,15}\subset\PP(1,3,4,5,7,8)$&I&$\frac{1}{7},15,537$; $\frac{1}{8},12,539$&$34$&$-64$&$37$\\
\evnrow 541&$X_{14,15}\subset\PP(1,3,4,5,7,10)$&I&$\frac{1}{10},10,539$&$36$&$-68$&$39$\\
540&$X_{14,16}\subset\PP(1,3,4,5,7,11)$&I&$\frac{1}{11},8,539$&$38$&$-72$&$41$\\
\evnrow 538&$X_{15,16}\subset\PP(1,3,4,5,8,11)$&I&$\frac{1}{11},10,537$&$39$&$-74$&$42$\\
535&$X_{20,21}\subset\PP(1,3,4,7,10,17)$&I&$\frac{1}{17},6,534$&$52$&$-100$&$54$\\
\evnrow 509&$X_{14,15}\subset\PP(1,3,5,6,7,8)$&I&$\frac{1}{8},14,508$&$32$&$-60$&$35$\\
453&$X_{12,14}\subset\PP(1,4,4,5,6,7)$&$T^1$&&$28$&$-52$&$32$\\
\evnrow 359&$X_{14,16}\subset\PP(1,4,5,6,7,8)$&$T^1$&&$29$&$-54$&$32$\\
358&$X_{12,20}\subset\PP(1,4,5,6,7,10)$&II$_1$&$\frac{1}{7},27,355$&$36$&$-68$&$39$\\
\evnrow 357&$X_{18,20}\subset\PP(1,4,5,6,9,14)$&I&$\frac{1}{14},8,356$&$38$&$-72$&$40$\\
354&$X_{18,20}\subset\PP(1,4,5,7,9,13)$&I&$\frac{1}{13},10,353$&$37$&$-70$&$38$\\
\evnrow 338&$X_{16,18}\subset\PP(1,4,6,7,8,9)$&$T^1$&&$30$&$-56$&$33$\\
297&$X_{18,20}\subset\PP(1,5,6,7,9,11)$&I&$\frac{1}{11},12,296$&$31$&$-58$&$32$\\
\evnrow 279&$X_{18,30}\subset\PP(1,6,8,9,10,15)$&$T^1$&&$36$&$-68$&$38$\\
265&$X_{24,30}\subset\PP(1,8,9,10,12,15)$&$T^1$&&$30$&$-56$&$31$\\
\evnrow 37&$X_{12,14}\subset\PP(2,3,4,5,6,7)$&$T^1$&&$18$&$-32$&$23$\\
\end{longtable}

{\ }

\begin{longtable}{cllccrc}
\caption{Codimension 3: $h^{1,1}(X)=1$ and $h^0(X,T_X)=0$ in all cases.}\label{tab!codim3}\\
\toprule
\multicolumn{1}{c}{grdb}&\multicolumn{1}{c}{variety}& \multicolumn{1}{c}{method} & \multicolumn{1}{c}{$\frac{1}r$, \#nodes, target id} & \multicolumn{1}{c}{$h^{2,1}$}&\multicolumn{1}{c}{$e(X)$}&\multicolumn{1}{c}{$h^1(T_X)$}\\
\cmidrule(lr){1-1}\cmidrule(lr){2-2}\cmidrule(lr){3-4}\cmidrule(lr){5-7}
\endfirsthead
\multicolumn{5}{l}{\vspace{-0.25em}\scriptsize\emph{\tablename\ \thetable{} continued from previous page}}\\
\midrule
\endhead
\multicolumn{5}{r}{\scriptsize\emph{Continued on next page}}\\
\endfoot
\bottomrule
\endlastfoot
\evnrow 26988&$X_{2,2...} = X_{2,2,2}\subset\PP^6$&\ \ $c_3(T_X)$&&14&$-24$&27\\
24077&$X_{2,3...}\subset\PP(1,1,1,1,1,1,2)$&\ \ I -- $T^1$&$\frac{1}{2},7,24076$&$14$&$-24$&$27$\\
\evnrow 20543&$X_{3,3...}\subset\PP(1,1,1,1,1,2,2)$&\ \ I -- I&$\frac{1}{2},8,20522$&$15$&$-26$&$28$\\
20523&$X_{3,3...}\subset\PP(1,1,1,1,1,2,3)$&\ \ I -- I&$\frac{1}{3},6,20522$&$17$&$-30$&$30$\\
\evnrow 16338&$X_{3,3...}\subset\PP(1,1,1,1,2,2,2)$&\ \ I -- I&$\frac{1}{2},10,16225$&$18$&$-32$&$31$\\
16226&$X_{3,4...}\subset\PP(1,1,1,1,2,2,3)$&\ \ I -- I&$\frac{1}{2},11,16204$; $\frac{1}{3},7,16225$&$21$&$-38$&$34$\\
\evnrow 16205&$X_{4,4...}\subset\PP(1,1,1,1,2,3,4)$&\ \ I -- I&$\frac{1}{4},7,16204$&$25$&$-46$&$38$\\
12062&$X_{4,4...}\subset\PP(1,1,1,2,2,2,2)$&\ \ I -- I&$\frac{1}{2},12,11435$&$15$&$-26$&$27$\\
\evnrow 11436&$X_{4,4...}\subset\PP(1,1,1,2,2,2,3)$&\ \ I -- I&$\frac{1}{2},14,11102$; $\frac{1}{3},8,11435$&$19$&$-34$&$31$\\
11122&$X_{4,4...}\subset\PP(1,1,1,2,2,3,3)$&\ \ I -- I&$\frac{1}{2},17,11002$; $\frac{1}{3},9,11102$&$24$&$-44$&$36$\\
\evnrow 11105&$X_{4,5...}\subset\PP(1,1,1,2,2,3,4)$&\ \ I -- I&$\frac{1}{2},18,10983$; $\frac{1}{4},8,11102$&$25$&$-46$&$37$\\
11103&$X_{4,5...}\subset\PP(1,1,1,2,2,3,5)$&\ \ I -- I&$\frac{1}{2},19,10982$; $\frac{1}{5},5,11102$&$28$&$-52$&$40$\\
\evnrow 11003&$X_{4,5...}\subset\PP(1,1,1,2,3,3,4)$&\ \ I -- I&$\frac{1}{3},11,10983$; $\frac{1}{4},9,11002$&$32$&$-60$&$44$\\
10984&$X_{5,6...}\subset\PP(1,1,1,2,3,4,5)$&\ \ I -- I&$\frac{1}{2},27,10961$; $\frac{1}{5},9,10983$&$34$&$-64$&$46$\\
\evnrow 10962&$X_{6,7...}\subset\PP(1,1,1,3,4,5,6)$&\ \ I -- I&$\frac{1}{6},11,10961$&$50$&$-96$&$62$\\
6859&$X_{4,5...}\subset\PP(1,1,2,2,2,3,3)$&\ \ I -- II$_1$&$\frac{1}{3},11,6858$&$21$&$-38$&$32$\\
\evnrow 5962&$X_{5,5...}\subset\PP(1,1,2,2,3,3,3)$&\ \ I -- I&$\frac{1}{3},12,5857$&$20$&$-36$&$30$\\
5865&$X_{5,6...}\subset\PP(1,1,2,2,3,3,4)$&\ \ I -- I&$\frac{1}{3},13,5843$; $\frac{1}{4},10,5857$&$22$&$-40$&$32$\\
\evnrow 5858&$X_{5,6...}\subset\PP(1,1,2,2,3,3,5)$&\ \ I -- I&$\frac{1}{3},14,5839$; $\frac{1}{5},6,5857$&$26$&$-48$&$36$\\
5844&$X_{6,6...}\subset\PP(1,1,2,2,3,4,5)$&\ \ I -- I&$\frac{1}{5},10,5843$&$25$&$-46$&$35$\\
\evnrow 5840&$X_{6,7...}\subset\PP(1,1,2,2,3,5,7)$&\ \ I -- I&$\frac{1}{7},6,5839$&$34$&$-64$&$44$\\
5515&$X_{6,6...}\subset\PP(1,1,2,3,3,3,4)$&\ \ I -- I&$\frac{1}{3},15,5261$; $\frac{1}{4},11,5514$&$22$&$-40$&$31$\\
\evnrow 5302&$X_{6,6...}\subset\PP(1,1,2,3,3,4,4)$&\ \ I -- I&$\frac{1}{3},17,5200$; $\frac{1}{4},12,5261$&$25$&$-46$&$34$\\
5267&$X_{6,7...}\subset\PP(1,1,2,3,3,4,5)$&\ \ I -- I&$\frac{1}{3},18,5161$; $\frac{1}{5},11,5261$&$26$&$-48$&$35$\\
\evnrow 5264&$X_{6,6...}\subset\PP(1,1,2,3,3,4,5)$&\ \ I -- I&$\frac{1}{3},19,5159$; $\frac{1}{4},13,5258$; $\frac{1}{5},7,5261$&$30$&$-56$&$39$\\
5262&$X_{6,7...}\subset\PP(1,1,2,3,3,4,7)$&\ \ I -- I&$\frac{1}{3},20,5158$; $\frac{1}{7},5,5261$&$32$&$-60$&$41$\\
\evnrow 5259&$X_{6,8...}\subset\PP(1,1,2,3,3,5,8)$&\ \ I -- I&$\frac{1}{3},23,5154$; $\frac{1}{8},5,5258$&$38$&$-72$&$47$\\
5201&$X_{6,7...}\subset\PP(1,1,2,3,4,4,5)$&\ \ I -- I&$\frac{1}{4},14,5161$; $\frac{1}{5},12,5200$&$30$&$-56$&$39$\\
\evnrow 5175&$X_{6,7...}\subset\PP(1,1,2,3,4,5,5)$&\ \ I -- I&$\frac{1}{5},13,5159$; $\frac{1}{5},8,5161$&$36$&$-68$&$45$\\
5162&$X_{7,8...}\subset\PP(1,1,2,3,4,5,6)$&\ \ I -- I&$\frac{1}{3},24,5138$; $\frac{1}{6},12,5161$&$32$&$-60$&$41$\\
\evnrow 5160&$X_{6,8...}\subset\PP(1,1,2,3,4,5,7)$&\ \ I -- I&$\frac{1}{4},17,5155$; $\frac{1}{7},7,5159$&$42$&$-80$&$51$\\
5139&$X_{8,9...}\subset\PP(1,1,2,4,5,6,7)$&\ \ I -- I&$\frac{1}{7},14,5138$&$42$&$-80$&$51$\\
\evnrow 4999&$X_{8,8...}\subset\PP(1,1,3,4,4,5,5)$&\ \ I -- I&$\frac{1}{4},19,4936$; $\frac{1}{5},15,4985$&$29$&$-54$&$36$\\
4988&$X_{8,9...}\subset\PP(1,1,3,4,4,5,6)$&\ \ I -- I&$\frac{1}{4},20,4912$; $\frac{1}{6},14,4985$&$30$&$-56$&$37$\\
\evnrow 4986&$X_{8,9...}\subset\PP(1,1,3,4,4,5,9)$&\ \ I -- I&$\frac{1}{4},23,4910$; $\frac{1}{9},5,4985$&$39$&$-74$&$46$\\
4937&$X_{8,9...}\subset\PP(1,1,3,4,5,5,6)$&\ \ I -- I&$\frac{1}{5},17,4912$; $\frac{1}{6},15,4936$&$33$&$-62$&$40$\\
\evnrow 4914&$X_{9,10...}\subset\PP(1,1,3,4,5,6,7)$&\ \ I -- I&$\frac{1}{4},25,4894$; $\frac{1}{7},15,4912$&$35$&$-66$&$42$\\
4913&$X_{8,9...}\subset\PP(1,1,3,4,5,6,7)$&\ \ I -- I&$\frac{1}{6},17,4911$; $\frac{1}{7},7,4912$&$43$&$-82$&$50$\\
\evnrow 4895&$X_{10,11...}\subset\PP(1,1,3,5,6,7,8)$&\ \ I -- I&$\frac{1}{8},17,4894$&$43$&$-82$&$50$\\
4849&$X_{10,11...}\subset\PP(1,1,4,5,6,6,7)$&\ \ I -- I&$\frac{1}{6},20,4837$; $\frac{1}{7},18,4848$&$37$&$-70$&$43$\\
\evnrow 4838&$X_{11,12...}\subset\PP(1,1,4,5,6,7,8)$&\ \ I -- I&$\frac{1}{5},27,4823$; $\frac{1}{8},18,4837$&$39$&$-74$&$45$\\
4824&$X_{12,13...}\subset\PP(1,1,4,6,7,8,9)$&\ \ I -- I&$\frac{1}{9},20,4823$&$46$&$-88$&$52$\\
\evnrow 4809&$X_{14,15...}\subset\PP(1,1,5,7,8,9,10)$&\ \ I -- I&$\frac{1}{10},23,4808$&$50$&$-96$&$55$\\
4796&$X_{16,17...}\subset\PP(1,1,6,8,9,10,11)$&\ \ I -- I&$\frac{1}{11},26,4795$&$54$&$-104$&$59$\\
\evnrow 2420&$X_{6,7...}\subset\PP(1,2,2,3,3,4,5)$&\ \ I -- II$_1$&$\frac{1}{5},8,2419$&$21$&$-38$&$29$\\
2404&$X_{9,10...}\subset\PP(1,2,2,3,5,7,9)$&\ \ I -- I&$\frac{1}{9},8,2403$&$32$&$-60$&$39$\\
\evnrow 1409&$X_{7,8...}\subset\PP(1,2,3,3,4,4,5)$&\ \ II$_1$&$\frac{1}{4},21,1389$&20&$-36$&27\\
1396&$X_{8,8...}\subset\PP(1,2,3,3,4,5,5)$&\ \ I -- I&$\frac{1}{5},10,1390$&$20$&$-36$&$26$\\
\evnrow 1394&$X_{8,9...}\subset\PP(1,2,3,3,4,5,7)$&\ \ I -- I&$\frac{1}{7},8,1390$&$22$&$-40$&$28$\\
1391&$X_{8,9...}\subset\PP(1,2,3,3,4,5,8)$&\ \ I -- I&$\frac{1}{8},6,1390$&$24$&$-44$&$30$\\
\evnrow 1252&$X_{8,9...}\subset\PP(1,2,3,4,4,5,5)$&\ \ I -- II$_1$&$\frac{1}{5},11,1249$&$20$&$-36$&$26$\\
1250&$X_{8,9...}\subset\PP(1,2,3,4,4,5,7)$&\ \ I -- II$_1$&$\frac{1}{7},7,1249$&$24$&$-44$&$30$\\
\evnrow 1184&$X_{8,9...}\subset\PP(1,2,3,4,5,5,6)$&\ \ I -- II$_1$&$\frac{1}{5},13,1171$&$24$&$-44$&$30$\\
1180&$X_{9,10...}\subset\PP(1,2,3,4,5,5,7)$&\ \ I -- I&$\frac{1}{5},13,1165$; $\frac{1}{7},9,1179$&$23$&$-42$&$28$\\
\evnrow 1168&$X_{9,10...}\subset\PP(1,2,3,4,5,7,7)$&\ \ I -- I&$\frac{1}{7},10,1164$; $\frac{1}{7},8,1165$&$28$&$-52$&$33$\\
1166&$X_{10,11...}\subset\PP(1,2,3,4,5,7,9)$&\ \ I -- I&$\frac{1}{9},9,1165$&$27$&$-50$&$32$\\
\evnrow 1157&$X_{10,12...}\subset\PP(1,2,3,5,5,7,12)$&\ \ I -- I&$\frac{1}{5},19,1150$; $\frac{1}{12},5,1156$&$33$&$-62$&$37$\\
1153&$X_{10,12...}\subset\PP(1,2,3,5,7,8,9)$&\ \ I -- I&$\frac{1}{8},9,1151$; $\frac{1}{9},11,1152$&$37$&$-70$&$41$\\
\evnrow 1090&$X_{12,13...}\subset\PP(1,2,5,6,7,7,8)$&\ \ I -- II$_1$&$\frac{1}{7},15,1083$&$27$&$-50$&$31$\\
1081&$X_{14,15...}\subset\PP(1,2,5,6,7,9,11)$&\ \ I -- I&$\frac{1}{11},12,1080$&$30$&$-56$&$33$\\
\evnrow 868&$X_{10,12...}\subset\PP(1,3,3,4,5,7,10)$&\ \ I -- I&$\frac{1}{10},7,867$&$25$&$-46$&$29$\\
641&$X_{10,11...}\subset\PP(1,3,4,4,5,6,7)$&\ \ I -- $T^1$&$\frac{1}{7},9,640$&$20$&$-36$&$24$\\
\evnrow 568&$X_{10,11...}\subset\PP(1,3,4,5,5,6,7)$&\ \ II$_1$&$\frac{1}{5},22,545$&21&$-38$&25\\
548&$X_{12,13...}\subset\PP(1,3,4,5,6,7,10)$&\ \ I -- I&$\frac{1}{10},8,547$&$23$&$-42$&$26$\\
\evnrow 543&$X_{12,14...}\subset\PP(1,3,4,5,7,8,11)$&\ \ I -- I&$\frac{1}{8},11,540$; $\frac{1}{11},7,542$&$28$&$-52$&$30$\\
510&$X_{14,15...}\subset\PP(1,3,5,6,7,8,11)$&\ \ I -- I&$\frac{1}{11},9,509$&$24$&$-44$&$26$\\
\evnrow 454&$X_{12,13...}\subset\PP(1,4,4,5,6,7,9)$&\ \ I -- $T^1$&$\frac{1}{9},8,453$&$21$&$-38$&$24$\\
392&$X_{12,13...}\subset\PP(1,4,5,5,6,7,8)$&$\ \ T^1$&&20&$-36$&23\\
\evnrow 326&$X_{14,15...}\subset\PP(1,5,5,6,7,8,9)$&$\ \ T^1$&&20&$-36$&22\\
298&$X_{16,17...}\subset\PP(1,5,6,7,8,9,10)$&$\ \ T^1$&&20&$-36$&22\\
\end{longtable}

\end{document}